\documentclass[nohyperref]{article}

\usepackage{bm}
\usepackage{microtype}
\usepackage{graphicx}
\usepackage{subfigure}
\usepackage{booktabs} 
\usepackage{tcolorbox}
\usepackage{relsize}
\usepackage{svg}
\usepackage{hyperref}

\usepackage[accepted]{icml2022}

\usepackage{amsmath}
\usepackage{amssymb}
\usepackage{mathtools}
\usepackage{amsthm}
\usepackage{enumitem}

\usepackage[capitalize,noabbrev]{cleveref}

\theoremstyle{plain}
\newtheorem{theorem}{Theorem}[section]

\newtheorem{lemma}[theorem]{Lemma}
\newtheorem{problem}[theorem]{Problem}
\newtheorem{corollary}[theorem]{Corollary}
\theoremstyle{definition}
\newtheorem{definition}[theorem]{Definition}

\theoremstyle{remark}
\newtheorem{remark}[theorem]{Remark}

\usepackage[textsize=tiny]{todonotes}

\icmltitlerunning{Approximate Frank-Wolfe Algorithms over Graph-structured Support Sets}

\newcommand{\argmin}[1]{\underset{#1}{\mathrm{argmin}}}
\newcommand{\argmax}[1]{\underset{#1}{\mathrm{argmax}}}
\makeatletter
\newtheorem*{rep@theorem}{\rep@title}
\newcommand{\newreptheorem}[2]{%
\newenvironment{rep#1}[1]{%
 \def\rep@title{#2 \ref{##1}}%
 \begin{rep@theorem}}%
 {\end{rep@theorem}}}
\makeatother
\newreptheorem{theorem}{Theorem}
\newreptheorem{corollary}{Corollary}
\newreptheorem{lemma}{Lemma}

\usepackage{pifont}
\newcommand{\cmark}{\ding{51}}%
\newcommand{\xmark}{\ding{55}}%
\definecolor{ForestGreen}{RGB}{34,139,34}

\DeclarePairedDelimiter\ceil{\lceil}{\rceil}

\begin{document}

\twocolumn[
\icmltitle{Approximate Frank-Wolfe Algorithms over Graph-structured Support Sets}

\begin{icmlauthorlist}
\icmlauthor{Baojian Zhou}{fudan}
\icmlauthor{Yifan Sun}{stonybrook}
\end{icmlauthorlist}

\icmlaffiliation{fudan}{School of Data Science, Fudan University, Shanghai, China}
\icmlaffiliation{stonybrook}{Department of Computer Science, Stony Brook University, Stony Brook, New York, USA}

\icmlcorrespondingauthor{Baojian Zhou}{bjzhou@fudan.edu.cn}

\icmlkeywords{Frank-Wolfe, Graph-structured, Dual Projection}

\vskip 0.3in
]

\printAffiliationsAndNotice{} %

\begin{abstract}
In this paper, we consider approximate Frank-Wolfe (FW) algorithms to solve convex optimization problems over graph-structured support sets where the \textit{linear minimization oracle} (LMO) cannot be efficiently obtained in general. We first demonstrate that two popular approximation assumptions (\textit{additive} and \textit{multiplicative gap errors)} are not applicable in that no cheap gap-approximate LMO oracle exists. Thus, \textit{approximate dual maximization oracles} (DMO) are proposed, which approximate the inner product rather than the gap. We prove that the standard FW method using a $\delta$-approximate DMO converges as $\mathcal{O}((1-\delta) \sqrt{s}/\delta)$ in the worst case, and as $\mathcal{O}(L/(\delta^2 t))$ over a $\delta$-relaxation of the constraint set. Furthermore, when the solution is on the boundary, a variant of FW converges as $\mathcal{O}(1/t^2)$ under the quadratic growth assumption. Our empirical results suggest that even these improved bounds are pessimistic, showing fast convergence in recovering real-world images with graph-structured sparsity.
\end{abstract}

\section{Introduction}
\label{submission}
This paper deals with the following graph-structured convex optimization (GSCO) problem 
\begin{equation}
\min_{\bm x \in \mathbb{R}^d} f(\bm x), \text{ subject to } \bm x \in \mathcal{D}(C,\mathbb{M}), \label{def:problem}
\end{equation}
where $\mathcal{D}(C, \mathbb{M}) \triangleq \operatorname{conv}\left\{\bm x: \|\bm x \|_2 \leq C, {\rm supp}\left({\bm x}\right) \in \mathbb{M} \right\}$ is a convex hull of the graph-structured support set described by $\mathbb{M}$, which contains a collection of allowed structures of the problem, and $f$ is a convex differentiable function. The support of $\bm x$, i.e., $\operatorname{supp}(\bm x) \triangleq \{i: x_i \ne 0\}$ encodes the sparsity pattern of $\bm x$, which can be defined by interesting graph structures such as a path, tree, or cluster over an underlying graph. Models $\mathbb{M}$ describe many interesting scenarios where graph structures serve as a powerful prior. Important applications of these include generalized $s$-sparsity \citep{argyriou2012sparse,lim2017k}, structured-sparsity \citep{bach2012optimization,bach2012structured}, clustered-sparsity \citep{mcdonald2016new}, weighted graph models (WGM) \citep{hegde2015nearly}, graph LASSO \citep{sharpnack2012sparsistency,hallac2015network}, marginal polytope \citep{krishnan2015barrier}, and many others  \citep{baraniuk2010model}.

To solve the GSCO problem, a natural idea is to use projected gradient descent (PGD) where, a \textit{projection oracle} finds a point in $\mathcal{D}$ at per-iteration. PGD-based methods for sparse and structure optimization have been well explored \cite{bahmani2013greedy,jain2014iterative,yuan2014gradient,nguyen2017linear,hegde2015approximation,hegde2015nearly,hegde2016fast}. To obtain approximate convergence guarantees, existing works of this type assume projection oracles can be solved exactly or with very high approximations guarantees. However, projections satisfying these requirements are usually hard to find for problem (\ref{def:problem}). Furthermore, multiple projections may be needed at per-iteration \cite{hegde2015nearly}.

Instead, the Frank-Wolfe (FW) algorithm \citep{frank1956algorithm} (a.k.a conditional gradient method) has been receiving increasing attention in recent years. Unlike PGD-based methods, FW-type methods, at each iteration, find a point using the \textit{linear minimization oracle} (LMO), which for many constraints may enjoy a much cheaper per-iteration cost than the projection oracle
\citep{combettes2021complexity}, and often obtain high-quality sparser solutions in early iterations. Hence, they are attractive for solving structured problems  \citep{krishnan2015barrier,briol2015frank,ping2016learning,berthet2017fast,allen2017linear,abernethy2017frank}. Yet, these attractive methods are less explored for graph-structured optimization problems.

To resolve problem (\ref{def:problem}) using FW-type methods, the main difficulty is that solving the LMO efficiently, even with $\mathcal D$ convex, is in general NP-hard for many structured models $\mathbb{M}$. A typical motivational example is a popular weighted graph model (WGM), where $\mathbb{M}$ contains all sets of $g$ connected components of a specified weighted graph \citep{hegde2015nearly}. Here, both the projection oracle and LMO are NP-hard to compute. While convergence rates exist for FW with approximate LMOs, they tend to be limited to two kinds of approximations: \textit{additive gap-approximate LMO} \citep{dunn1978conditional,jaggi2013revisiting} and \textit{multiplicative gap-approximate LMO} \citep{locatello2017unified,pedregosa2020linearly}. As such, we ask the following crucial question

\begin{center}
\textit{Do an additive and multiplication gap-approximate LMO exist for solving GSCO problems?}
\end{center}

\paragraph{Our contributions.} In particular, by answering this question negatively, we open and explore the space for inexact FW methods that are more appropriate for this class of problems. As demonstrated, for a WGM $\mathbb{M}$ of $\mathcal{D}$, one can always find adversarial examples to show that gap-additive and gap-multiplicative LMOs are as hard to resolve as exact LMOs. Therefore, the existing approximate-LMO FW convergence rates are inapplicable to GSCOs in general.\footnote{Although this paper focuses on a specific model, our proposed methods are applicable for any $\mathbb{M}$ whenever its DMO is available.}

Instead, we propose to use an approximate \textit{dual maximization oracle} (DMO), which for several important GSCO problems can be easy to find in practice. This assumption is equivalent to multiplicatively approximating a key inner product, rather than the gap. We show that, for GSCOs, a simple heuristic method acts as an approximate DMO with a constant $\delta$ error guarantee,  but no non-exact gap-additive or gap-multiplicative LMO exists.

The main theoretical contribution is then to give convergence rates of the approximate FW methods under a $\delta$-approximate DMO assumption. We show that when $f$ is $L$-smooth,  a standard FW using a $\delta$-approximate DMO converges as $\mathcal{O}((1-\delta)\sqrt{s} / \delta)$ where $s$ is the  maximum sparsity allowed in $\mathbb{M}$, and as $\mathcal{O}(L/(\delta^2 t))$ over $\mathcal{D}/\delta$. The convergence rate of the latter case is consistent with recent advances of generalized matching pursuit (MP) \citep{locatello2018matching}.

Inspired by works concerning the \textit{nearest extreme point oracle} \cite{garber21a}, we propose a new variant of FW, which is  empirically faster. We also show that the convergence rate of this new variant is faster at $\mathcal{O}(1/t^2)$ when the solutions are on the boundary with $\delta=1$. Empirically, we observe that these assumptions are not necessary to achieve this faster rate. Additionally, we show that an approximate version converges to a point in $\mathcal{D}/\delta$ at $\mathcal{O}(1/t)$.

The main contributions can be summarized as follows:

\begin{itemize}
\item We prove that the exact DMO for general GSCOs is NP-hard and no efficient additive and multiplicative gap-approximate LMO exists.
\item We propose two FW-type methods and provide convergence rate analysis when the problem is solved over $\mathcal{D}$ or a relaxed set $\mathcal{D}/\delta$.
\item We empirically demonstrate that the proposed methods are more effective and efficient compared with PGD-based and MP-based methods on the graph-structured linear regression problem.
\end{itemize}

\subsection{Related work}

Existing works on sparse optimization considers models that limit the number of nonzeros in the solution. However, many real-world applications have more intricate graph structures as important priors; the tradeoff is that strictly enforcing them often loses convexity of the feasibility space.

\paragraph{FW method and its variants.} The FW method \citep{frank1956algorithm} and its variants \citep{jaggi2013revisiting,lacoste2015global,garber2016linear,bashiri2017decomposition,balasubramanian2018zeroth,kerdreux2018frank,lei2019primal,luise2019sinkhorn,locatello2019stochastic,thekumparampil2020projection,garber2020revisiting,combettes2020boosting,pedregosa2020linearly,sun2020safe}  for convex constrained problems have recently received popularity mainly due to two advantages. First, it is projection free--the LMO is often much cheaper to compute than the projection oracle. Second, in applications with desired structured sparsity, early FW iterations tend to be naturally sparse. Inspired from these advantages, we seek to propose FW-type methods for GSCO problems.

Recent works put effort into accelerating FW with modifications. More specifically, \citet{lacoste2015global} and \citet{garber2016linear}
propose away-step variants to reduce the computation overheads. An important related work \citep{garber2015faster} shows that if $\mathcal D$ is a strongly convex set, then the FW rate can be improved to $\mathcal{O}(1/t^2)$. However, sparsity-inducing sets are generally not strongly convex, favoring low-dimensional facets of $\mathcal D$ as solutions; this is also true for our graph-sparsity application. We therefore achieve the $O(1/t^2)$ acceleration without necessitating strong convexity, drawing inspiration from the \emph{nearest extreme point oracle} explored in \citet{garber21a}.

\paragraph{Connections with other methods.} One of the reasons that FW is so heavily studied is its connections with other important greedy methods. For example, \citet{bach2015duality} shows that FW is closely related with mirror decent through duality; and \citet{locatello2017unified} and \citet{combettes2019blended} explore the close connection between FW and MP. 

\paragraph{Approximation of LMO.} The study of inexact FW methods tend to center on two types of LMO errors: gap-additive \citep{dunn1978conditional,jaggi2013revisiting} and gap-multiplicative \citep{locatello2017unified,pedregosa2020linearly}. Under these two assumptions, the convergence rates are $\mathcal{O}(\delta/t)$ and $\mathcal{O}(1/(\delta^2 t))$ respectively. However, as we will demonstrate in Sec. \ref{section:gap-lmo}, these two regimes do not adequately describe efficient methods for GSCO problems. Instead, we explore that is multiplicative with respect to an inner product; this is inspired by \cite{locatello2018matching} making a similar analysis for MP, and is related to the works of \citet{hazan2018online} and \citet{garber2017efficient}. Another related work is \citet{kerdreux2018frank}, which analyzes FW methods where the LMO is approximated by subsampling a subset of atoms at each iteration.

\section{Preliminaries}
\label{gen_inst}
We begin by introducing notations and setup. We then define the graph-structured support sets and the FW-type algorithm. \textit{All proofs are postponed to the appendix.}

\paragraph{Notations and setup.} We consider optimization over variables $\bm x\in \mathbb R^d$, which induce a ground set $[d]:=\left\{1,2,\ldots, d\right\}$. Uppercase letters (e.g., $I, S$) stand for subsets of $[d]$. ${\bm A} \in \mathbb{R}^{n\times d}$ denotes a matrix and ${\bm x}, {\bm y} \in \mathbb{R}^d$ are column vectors. The masked vector $\bm x_S$ is defined as ${\bm x}_S(i) = x_i$ if $i \in S$, and 0 otherwise. We consider problems in  Euclidean space equipped with an inner product $(\mathbb{R}^d, \langle\cdot,\cdot\rangle)$, where $\langle \bm x, \bm y\rangle := \sum_{i=1}^d x_i y_i$. The induced $\ell_2$-norm is $\|\bm x\|_2 = \sqrt{\langle \bm x,\bm x \rangle}$.  An $1/\alpha$-scaling of set $\mathcal{D}$ is denoted as $\mathcal{D}/\alpha := \{ \bm x/\alpha : \bm x \in \mathcal{D}\}$. 

We define $\bm x$ on an underlying  graph $\mathbb{G} (\mathbb{V},\mathbb{E})$ where the node set $\mathbb{V}=[d]$ and edge set $\mathbb{E} \subseteq [d]\times [d]$. Each entry $x_i$ is associated with each node $v_i$. Given $S^\prime \subseteq \mathbb{V}, \mathbb{E}^\prime \subseteq \mathbb{E}$, $\mathbb{G}(S^\prime,\mathbb{E}^\prime)$ represents a subgraph of $\mathbb{G}$, we say $f$ is $L$-smooth if, for all ${\bm x},{\bm y} \in \mathcal{D}$, there exists $L>0$ such that
\begin{equation*}
f({\bm x}) - f({\bm y}) - \langle \nabla f({\bm y}), {\bm x} -{\bm y} \rangle \leq \frac{L}{2} \| {\bm x} -{\bm y} \|_2^2.
\end{equation*}

\paragraph{Graph-structured support sets.} Before introducing the graph-structured support sets, we recall a more general form, structured support sets, as defined in the following.

\begin{definition}[Structured support set]
Given $\mathbb{M}:=\{S_1, S_2, \ldots, S_m\}$, a collection of subsets of $[d]$, with $\cup_{i} S_i = [d]$, a structured support set $\mathcal{D}(C,\mathbb{M})$ is
\begin{equation}
\mathcal{D}(C,\mathbb{M}) := 
\operatorname{conv}\{\bm x: \|\bm x\|_2 \leq C, {\rm supp}({\bm x}) \in \mathbb{M} \}. \nonumber
\end{equation}
\end{definition}
The definition above is generally enough to model many interesting structures of $\bm x$. However, for graph-related problems, we specifically focus on where each $S_i$ captures graph-structured information important for real-world applications. We define the graph-structured support sets as the following.
\begin{definition}[Graph-structured support sets $\mathcal{D}$] Define a graph $\mathbb{G}(\mathbb{V}=[d], \mathbb{E})$
and consider a structured support set $\mathcal{D} := \mathcal{D}(C, \mathbb{M})$ where $C>0$ and $\mathbb{M}= \{S_1, S_2,\ldots, S_m\}$. $\mathcal{D}$ is a \textit{graph-structured support set} if each $S_i$ is associated with a subgraph $\mathbb{G}(S_i,\mathbb{E}_i)$. We simply denote this set as $\mathcal{D}$.
\label{def:graph-support-norm-set}
\end{definition}

\begin{figure}
\centering
\begin{minipage}{0.25\textwidth}
\centering
\includegraphics[width=0.6\textwidth]{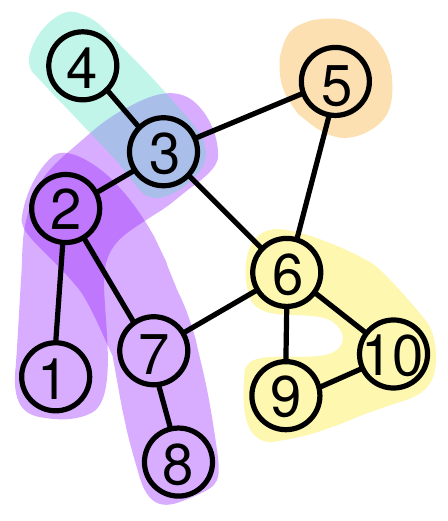}
\vspace{-5mm}
\end{minipage}%
\begin{minipage}{0.25\textwidth}
\centering
\includegraphics[width=0.6\textwidth]{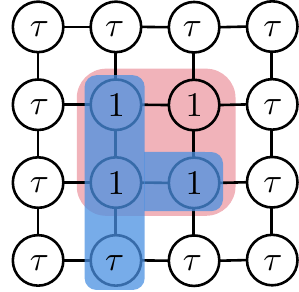}
\vspace{-5mm}
\end{minipage}
\caption{\textit{Left}: An element of a $g$-subgraph model $\mathbb{M}(\mathbb{G}, s,g)$  defined on a 10-node graph where each colored region is a subgraph. ($g=4$, $s=5$).  \textit{Right}: An example graph with $\mathbb{M}=\mathbb{M}(\mathbb{G}, s=4, g=1)$ where $f(\bm x) = \bm x^\top \bm x/2 - \bm x^\top \bm b$ with $\bm b = [1, 1, 1, 1, \tau, \ldots, \tau]^\top$ for some $0 < \tau < 1/2$ and . The red region is the optimal structure support of $\bm x^*$ while the blue is a best approximate structure support.\vspace{-5mm}}
\label{fig:example-model}
\end{figure}

Models of many real-world learning problems can be expressed as elements of the graph-structured support set. For example, learned structured polytopes are connected paths in a graph \cite{garber2016linear}, and models of multi-task learning are connected cliques \cite{mcdonald2016new}. An important instance of Def. \ref{def:graph-support-norm-set} is a simplified version of the weighted graph model $\mathbb{M}$ \cite{hegde2015nearly}, which we call a $g$-subgraph model.
\begin{definition}[$g$-subgraph model~\citep{hegde2015nearly}]
Given an underlying  graph $\mathbb{G}(\mathbb{V},\mathbb{E})$,  the graph model $\mathbb{M}\triangleq \mathbb{M}(\mathbb{G}, s, g)$ is a  union of $g$ subgraphs
\begin{equation}
\mathbb{M}(\mathbb{G}, s, g) = \{ S \triangleq S_1 \cup \cdots \cup S_g: |S|\leq s \}, \label{def:wgm}
\end{equation}
where subgraph $\mathbb{G}(S_i,\mathbb{E}_i)$ is connected.
\end{definition}
We illustrate $\mathbb{M}(\mathbb{G}, s, g)$ in Fig. \ref{fig:example-model} (left).

\paragraph{FW-type algorithm.}  Given an initial $\bm x_0 \in \mathcal{D}$ and a learning rate $\eta_t \in [0,1]$, for all $t \geq 0$, a FW-type algorithm for solving problem (\ref{def:problem}) uses the following updates 
\begin{equation}
\bm x_{t+1} = \bm x_t + \eta_t (\rho\bm v_t - \bm x_t), 
\label{equ:fw-updates}
\end{equation}
where $\bm v_t$ is the minimizer of \textit{graph-structured LMO}
\begin{equation}
\text{LMO}: \bm v_t \in \argmin{\bm v \in \mathcal{D}(C,\mathbb{M})} \left\langle a \bm x_t + b \nabla f(\bm x_t), \bm v \right\rangle. \label{equ:lmo}    
\end{equation}
Setting $a=0, b = 1, \eta_t = 2 / (t+2)$ with $\rho = 1$ recovers the standard FW. We will discuss more of its details in Sec. \ref{section:main-algorithms}. In the rest, we focus on the graph-structured support set $\mathcal D(C, \mathbb M)$ corresponding to $g$-subgraph model $\mathbb M(\mathbb{G}, s, g)$.

\section{Graph-structured LMOs}
\label{section:gap-lmo}
This section deals with the computation aspects of our generalized LMO (\ref{equ:lmo}). We first present the hardness result in the standard case and demonstrate how two popular gap-based LMOs are inappropriate in our setting. 

\subsection{Hardness of graph-structured LMOs}
\label{section:hardness-result}
The computational barrier of  (\ref{equ:lmo}) is mainly due to the combinatorial nature of $\mathbb{M}$. To illustrate this, we first limit our attention to the familiar LMO, where  $a=0, b=1$. Note that in general, the LMO returns $\bm v_t$ an extreme point of $\mathcal D$. Hence, one can reformulate it as a subspace identification problem\footnote{Without loss of generality, we assume $a=0, b=1$ in (\ref{equ:lmo}). In the rest, we always assume $\mathcal{D}:= \mathcal{D}(C,\mathbb{M})$.}
\[
\min_{\bm v \in \mathcal D} \left\langle \nabla f(\bm x_t), \bm v \right\rangle = \min_{S^* \in \mathbb{M}, \|\bm v\|_2\leq C} \left\langle \nabla f(\bm x_t)_{S^*}, \bm v  \right\rangle,
\]
where $S^*$ is an optimal support set minimizing the inner product in \eqref{equ:lmo}. Hence, the minimizer is
\begin{equation} \bm v_t = - \tfrac{C \cdot \nabla f(\bm x_t)_{S^*}}{\| \nabla f(\bm x_t)_{S^*}\|_2}, \; S^* \in \argmax{S \in \mathbb{M}}\, \| \nabla f(\bm x_t)_S \|_2^2.
\label{equ:lmo-solution}
\end{equation}
From \eqref{equ:lmo-solution}, we see that the computational complexity of graph-structured LMO is bottlenecked by finding $S^*$, whose complexity depends solely on $\mathbb{M}$. For example, in the simple instance when $\mathbb{M} = \{S \subseteq [d] : |S| \leq s \}$, the optimal $s$-sparse set  can be discovered in $\mathcal{O}(d)$ time using a selection algorithm \cite{floyd1975algorithm}.
However, when $\mathbb{M}$ is the $g$-subgraph model (\ref{def:wgm}) or other types of models \cite{baldassarre2016group,lim2017k}, finding $S^*$ is NP-hard in general as presented in Theorem \ref{thm:np-hard}.
\begin{theorem}[Hardness of graph-structured LMO]
Given $C >0, a=0, b=1$, convex differentiable function $f$, and the graph-structured model $\mathbb{M}$, the computation of an exact graph-structured LMO defined in (\ref{equ:lmo}) is NP-hard. 
\label{thm:np-hard}
\end{theorem}
\begin{remark}
The NP-hardness result is obtained via a reduction from the model selection problem \cite{baldassarre2016group}, which contains all instances of the maximum weighted coverage problem \cite{khuller1999budgeted}.
\end{remark}

\subsection{Approximate LMOs and adversarial examples}
\label{subsection:adv-example}
To overcome the computational barrier, we resort to finding an approximate $\bar{\bm v}_t$ at each iteration. To obtain meaningful convergence rates for FW-type algorithms, one must first establish a metric of \emph{approximation quality}, which will characterize the final bounds. 

\paragraph{Additive approximate LMO.} The gap-additive \footnote{Given $\bm x \in \mathcal{D}$, the ``gap'' is to measure the difference between $\min_{\bm v \in \mathcal{D}} \langle \nabla f(\bm x), \bm v \rangle$ and $\langle \nabla f(\bm x), \bar{\bm v}_t \rangle$} approximate LMO \cite{dunn1978conditional} finds $\bar{\bm v}_t$ such that
\begin{flalign}
\left\langle \nabla f(\bm x_t), \bar{\bm v}_t \right\rangle \leq \min_{\bm v \in \mathcal{D}} \left\langle \nabla f(\bm x_t), \bm v  \right\rangle + \mathcal{O}(\tfrac{\epsilon}{t}), \label{inequ:approx-lmo-1}
\end{flalign}
where $\epsilon$ is the accuracy parameter and the approximate tolerance must decay over time in order to obtain an $\mathcal{O}(1/t)$ convergence rate.

\paragraph{Multiplicative approximate LMO.} Another common approximation regime is the gap-multiplicative approximate LMO \citep{locatello2017unified}, which returns $\bar{\bm v}_t$ such that
\begin{flalign}
\underbrace{\left\langle \nabla  f(\bm x_t), \bar{\bm v}_t - \bm x_t \right\rangle}_{\bar g_t(\bm x_t)} \leq \delta \cdot \underbrace{\min_{\bm v \in \mathcal{D}} \left\langle \nabla f(\bm x_t), \bm v - \bm x_t \right\rangle}_{g_t(\bm x_t)}, \label{inequ:approx-lmo-2}
\end{flalign}
where $\delta \in (0,1]$ is the approximation factor. Note that \eqref{inequ:approx-lmo-2} is only possible if the approximation $\bar g_t(\bm x_t) < 0$ when $\bm x_t \notin \arg\min_\mathcal{D} f$. However, we show these two approximate LMOs \eqref{inequ:approx-lmo-1} and \eqref{inequ:approx-lmo-2} are generally impractical to obtain by constructing adversarial examples.  That is, applied to the graph-structured support set $\mathcal{D}$; cases exist where any ``approximation'' is necessarily exact.

\paragraph{An adversarial problem setup.} Consider a grid graph $\mathbb G$ (Fig. \ref{fig:example-model}, right) and assume that $\mathbb{M} = \mathbb{M}( \mathbb{G}, s=4, g=1)$, i.e., a set of connected components of $\mathbb{G}$, each with at most four nodes. We map the four center nodes to the first 4 coordinates of $\bm x$. Suppose $\tau \in (0,1/2)$ and
\[
f(\bm x) = \tfrac{1}{2} \bm x^\top \bm x - \bm x^\top \bm b, \quad \bm b = [1, 1, 1, 1,\tau, \cdots, \tau]^\top.
\]
Take $C = 1$ and then unique optimal solution of (\ref{def:problem}) has
\begin{eqnarray*}
\bm x^* &=&[\tfrac{1}{2},\tfrac{1}{2},\tfrac{1}{2},\tfrac{1}{2}, 0, \ldots, 0]^\top,\\
\nabla f(\bm x^*) &=& -[\tfrac{1}{2},\tfrac{1}{2},\tfrac{1}{2},\tfrac{1}{2},\tau,\cdots,\tau]^\top.
\end{eqnarray*}
When $\bm x_t = \bm x^*$, $\bm v_t = \bm x^*$, and the duality gap $g_t(\bm x^*)=0$.

\paragraph{Gap-additive bound cannot decay properly.} In the ideal case,  $\bm x_t \overset{t \rightarrow \infty}{\longrightarrow} \bm x^*$. However, in the example above,
pick any non-optimal $\bar{\bm v}_t$; for example, an instance, which only 1 element is suboptimal is \[\bar {\bm v}_t=\tfrac{1}{\sqrt{3/4+\tau^2}}[1,1,1,0,\tau, 0, \cdots, 0]^\top
\]
giving $g_t(\bm x^*) - \bar{g}_t({\bm x}^*) = 1- \sqrt{\tfrac{3}{4} + \tau^2}  > 0$,
which is strictly positive and constant in $t$, hence violating (\ref{inequ:approx-lmo-1}), even at $\bm x^*$. Moreover, this is a lower bound, since any more suboptimal $\bar {\bm v}_t$ can only increase this bound. The adversarial example at $\bm x^*$ shows that even if the algorithm starts arbitrarily close to the true solution $\bm x^*$, it still may fail with approximate LMOs. It is definitely not a unique problem point and not hard to extend this example not to $\bm x^*$ but any $\bm x_t$, for the appropriate problem.  For example, one can find two adversarial examples at a non-optimal point $\bm x_t$ in Appendix \ref{appendix:adv-examples}. Thus, any gap-additive assumption with a decaying tolerance will eventually require an exact LMO, which is NP-hard in general as stated in Thm. \ref{thm:np-hard}.

\paragraph{Gap-multiplicative estimate could be negative.} 
Continuing with the above example, we notice that
\[
\bar{g}_t(\bm x^*) = \nabla f(\bm x^*)^\top (\bm x^* - \bar{\bm v}_t) = \sqrt{3/4 + \tau^2} - 1 < 0,
\]
which means no positive value of $\delta \in (0, 1)$ can possibly satisfy (\ref{inequ:approx-lmo-2}); that is, the assumption is only satisfied if $\delta = 1$ and the LMO is exact, which again is NP-hard. Finally, we stress that the above adversarial examples are not contrived; they are generated using a very simple $f$ and can be applied to very general types of graphs.

To further strengthen the argument that there exist interesting cases where no efficient approximate algorithms exist, note that one can map all instances of the maximum weighted coverage problem \cite{khuller1999budgeted} to our problem, but the best greedy algorithm for the maximum weighted coverage problem has approximation threshold of $1-1/e$ \cite{feige1998threshold}. This
indicates a subset problems of (\ref{def:problem}) cannot have a better approximation bound.

\section{Approximate FW-type methods}
\label{section:main-algorithms}

In practice, the above  examples show that both (\ref{inequ:approx-lmo-1}) and (\ref{inequ:approx-lmo-2}) can be impossible to find unless we solve the LMO with exact support. Instead, rather than approximating the gap, we turn to approximating $\langle \nabla f(\bm x_t), \bm v \rangle$ using a \emph{dual maximization oracle}; here, efficient inexact methods do exist and are easier to obtain. We first present the dual maximization oracle and then present our proposed approximate methods.

\subsection{Approximate dual maximization oracle}
\label{section:dmo:inner-product}

\begin{definition}[Inner Product Operator]
Given $\bm z \in \mathbb{R}^d$, $\mathcal{D} \subseteq \mathbb{R}^d$, and approximation factor $\delta \in (0,1]$, the approximated Inner Product Operator (IPO) returns $\bm v$ such that
\begin{flalign}
\textit{Approximate IPO\quad} \left\langle \bm z, \bm v \right\rangle \leq \delta \cdot \min_{\bm s \in \mathcal{D}} \left\langle \bm z, \bm s \right\rangle. \label{def:lmo}
\end{flalign}
We denote such set of $\bm v$ as $(\delta, \bm z, \mathcal{D})$-IPO. 
\end{definition}
Note that in general, if $\mathcal D$ is a symmetric nontrivial set then this quantity is negative. This approximate inner product operator is not new. Outside of FW-type methods, it is an important operator that can be used in approximate Matching Pursuit algorithms  \citep{locatello2017unified,mokhtari2018conditional} and online linear optimization methods \citep{garber2017efficient,garber2021efficient}.

\begin{definition}[Dual Maximization Oracle]
Given the structure support set $\mathbb{M}$, the \textit{dual approximation oracle} finds an $S \in \mathbb{M}$  such that
\begin{equation}
\| \bm z_S\|_* \geq \delta \cdot\max_{S^\prime \in \mathbb{M}}\| \bm z_{S^\prime} \|_*, \label{inequ:approx-oralce}
\end{equation}
where $\delta \in (0,1]$ is the approximation factor and $\|\cdot \|_*$ corresponds to the dual norm ($\|\cdot\|_*=\|\cdot\|_2$ in our case). We denote such set $S$ as the $(\delta,\bm z, \mathcal{D}$)-DMO.
\end{definition}
An important family of examples of this oracle are the approximate projections studied in \citet{hegde2015nearly,hegde2016fast,golbabaee2018inexact}. A key observation from the above definitions is that the approximate IPO is equivalent to an approximate dual norm calculation. It is also in a recent work \citep{garber21a}: for certain norm balls, LMO is equivalent to the projection on extreme points. 

One may notice that the exact LMO and DMO
are identical. However, the approximation conditions are different: the usual assumptions for approximate LMO are not
always realistic which we demonstrated this in previous section. In contrast, the DMO-inspired condition is weaker, but always satiable with existing methods (see the following Section \ref{subsect:dmo}). This weaker assumption complicates the convergence analysis (see Section \ref{section:convergence-rate}), and we believe that the estimated error does not converge globally without making any assumptions.

\begin{theorem}
Given the set $\mathcal{D}$ and suppose $S \in (\delta,\bm z, \mathcal{D})$-DMO. Define the approximate supporting vector $\tilde{\bm v}_t \triangleq -C\cdot\bm z_S / \|\bm z_S\|_2$.  Then, $\tilde{\bm v}_t \in (\delta, \bm z, \mathcal{D})$-IPO.
\label{theorem:dmo-to-ipo}
\end{theorem}
Thm. \ref{theorem:dmo-to-ipo} provides an important way to obtain an approximate IPO via DMO, which is usually easier to obtain. That is, if one can find $S$ such that $S \in (\delta,\bm z, \mathcal{D})\text{-DMO}$, then $\tilde{\bm v}_t$ is guaranteed in $(\delta, \bm z, \mathcal{D})$-IPO.

\subsection{Efficient methods for DMO}
\label{subsect:dmo}
\paragraph{Top-$g$ + neighbor visiting.}
Let us first consider a simple method where we find an approximate $(\delta,z,\mathcal D)$-DMO \eqref{inequ:approx-oralce} for the $g$-subgraph model defined by \eqref{def:wgm} by first finding the indices of the $g$ largest magnitude elements (denote as $I_g\subset S$) in the dual vector $\bm z$. Then add any feasible additional elements to $S$ until $|S|=s$.\footnote{The exact algorithm description is in Appendix \ref{appendix:section:dmo}.} This can be done by just picking any adjacent nodes (without violating connectivity constraint) to the first $g$ ``seed" nodes. Then for any $S^\prime\in\mathbb M$,
\begin{align*}
\ceil*{\frac{s}{g}} \|\bm z_{S} \|_2^2 &= \ceil*{\frac{s}{g}} ( \sum_{i\in I_g} |z_i|^2 + \sum_{j\in S \backslash I_g} |z_j|^2 ) \geq \|\bm z_{S^\prime}\|^2_2;
\end{align*}
thus $S$ approximates $S^*$ with $\delta =\sqrt{ 1/\lceil s/g \rceil}$. This algorithm runs in $\mathcal{O}(|\mathbb{E}|)$.

In the adversarial example shown in Sec. \ref{subsection:adv-example}, the 1-wrong examples of $S^\prime$ also achieve an approximate DMO-operator with $ \delta = 3/4 +\tau^2$. In other words, unlike the gap-additive and gap-multiplicative approximations, \emph{simple approximate solutions do not need to be exact in order to satisfy the IPO/DMO conditions}. Furthermore, one can find better methods from existing projection methods. For example, we can use the \emph{head projection} \cite{hegde2015nearly}, which provides a $\delta$-DMO with $\delta = \sqrt{1/14}$ of a general weighted graph model, which runs in  polynomial time $\mathcal{O}(|\mathbb{E}| \log^3 d)$. 

\subsection{Approximate FW-type methods via DMO}

We present ``standard'' approximate FW-type methods in Alg. \ref{algo:approx-fw}. There are 2 different design choices, which results in 4 variations. First, the approximate DMO can be performed on $\bm z_t = -\nabla f(\bm x_t)$, as in the usual FW (termed \textsc{DMO-FW}), or on $\bm z_t = -(\bm x_t-\tfrac{1}{L\eta_t}\nabla f(\bm x_t))$ (termed \textsc{DMO-AccFW}), which gives an accelerated variant based on the nearest extreme point oracle explored in \citet{garber21a}; Second, to obtain $\bm x_t$, Option I (line 7) finds points within $\mathcal{D}$; while Option II (line 8) finds points over a relaxed constraint set $\mathcal{D}/\delta$. Specifically, \textsc{DMO-AccFW} returns a convex combination of normalized supporting ``gradient descent'' $S_t$. One may treat it as a “hybrid” of FW and gradient descent. Hence, it has project-free property like FW but converges faster than the usual FW method.

\begin{algorithm}[H]
\caption{FW-type methods for GSCOs}
\begin{algorithmic}[1]
\STATE \textbf{Input:} step size $\{\eta_t\}$, $\delta$, $L$, $C$, and $\mathbb{M}$
\STATE pick any point $\bm x_0$ in $\mathcal{D}$
\FOR{$t = 0,1,\ldots,$}
\STATE $\bm z_t = \begin{cases}\textsc{DMO-FW}:= - \nabla f(\bm x_t) \\
\textsc{DMO-AccFW}:= - \left(\bm x_t - \tfrac{\nabla f(\bm x_t)}{L \eta_t}\right)\end{cases}$
\STATE $S_t = (\delta, -\bm z_t, \mathcal{D})$-DMO
\STATE $\tilde{\bm v}_t =\tfrac{ C \cdot (\bm z_t)_{S_t}}{\left\|( \bm z_t)_{S_t}\right\|_2}$
\STATE Option \textsc{I}:  $\bm x_{t+1} = \bm x_t + \eta_t (\tilde{\bm v}_t - \bm x_t)$
\STATE Option \textsc{II}:  $\bm x_{t+1} = \bm x_t + \eta_t (\tilde{\bm v}_t / \delta - \bm x_t)$ 
\ENDFOR
\STATE \textbf{Return} $(\bm x_{\bar{t}}, f(\bm x_{\bar{t}})), \bar{t} \in \argmin{t} f(\bm x_t)$
\end{algorithmic}
\label{algo:approx-fw}
\end{algorithm}

\section{Convergence analysis and sensing results}
\label{section:convergence-rate}

Denote the primal error $h(\bm x_{t}) := f(\bm x_{t}) - f(\bm x^*)$ and assume the step size $\eta_t = 2/(t+2)$.\footnote{Note that the solution $\bm x_{\bar{t}}$ always has $h(\bm x_{\bar{t}}) \leq h(\bm x_t)$.} We first establish the convergence rates of the methods, and then showcase resulting on the graph-structured linear sensing problem. Throughout this section, we assume $f$ to be $L$-smooth and denote
$\bm x^* \in \arg\min_{\bm x\in \mathcal{D}} f(\bm x)$ as a minimizer of (\ref{def:problem}).

\subsection{Convergence rate of \textsc{DMO-FW}}
By $L$-smoothness, for all $t \geq 0$, \textsc{DMO-FW-I} admits
\[
h({\bm x_{t+1}}) \leq  \left(1-\delta\eta_t\right)h({\bm x_t}) + Q_t,
\]
where 
\[
Q_t \triangleq (1-\delta)\eta_t\langle \bm x_t, -\nabla f(\bm x_t) \rangle  + \tfrac{L \eta_t^2}{2} \|\tilde{\bm v}_t - \bm x_t\|_2^2.
\]
Clearly, if $\delta = 1$, $Q_t = \mathcal{O}(1/t^2)$ and thus we can show recursively that $h(\bm x_{t+1}) = \mathcal{O}(1/t)$. However, if $\delta < 1$, the convergence of $h$ is dominated by $\eta_t|\langle \bm x_t, \nabla f(\bm x_t)\rangle|$, which is not easy to bound in general. Thus, when $\delta < 1$, the standard proof technique fails. Instead, we first establish a $\delta$-dependent convergence rate.

\begin{theorem}[Universal rate of \textsc{DMO-FW-I}] For all $t \geq 1$, the primal error of \textsc{DMO-FW-I} satisfies
\[
h(\bm x_t) \leq \begin{cases}
\min\left\{ 2 C\sqrt{s} \|\nabla f(\bm x_t)\|_\infty, P(\delta,2\delta) \right\}, &\hspace{-2ex} \delta \in (0,\tfrac{1}{2}] \\
\min \left\{ P(\delta, 1),  \tfrac{3(1-\delta)h(\bm x_0) + A_t}{(2\delta-1)(t+2)}\right\}, &\hspace{-2ex} \delta \in (\tfrac{1}{2},1],
\end{cases}
\]
where $s$ is maximal allowed sparsity, $P(\delta,\nu)$ is
\begin{equation}
P(\delta,\nu)\triangleq\frac{(1-\delta)9^\delta}{(t+2)^{2\delta}} \cdot h(\bm x_0) + \frac{\ln(t+1) + 1}{(t+2)^{\nu}} \cdot A_t \nonumber
\end{equation}
and $A_t \triangleq \max_{i \in [t]} 4(1-\delta)(i+2) |\left\langle \bm x_i, \nabla f(\bm x_i) \right\rangle| + 8 L C^2$.
\label{thm:1:fw-dmo-i}
\end{theorem}

When $\delta = 1$, Thm. \ref{thm:1:fw-dmo-i} recovers the standard convergence rate of the exact FW method. In the approximate case $(\delta < 1)$, the bound is more involved. With an additional assumption on the decay of the magnitude of gradient, the following corollary shows the overall practical bound.

\begin{corollary}[Practical rate of \textsc{DMO-FW-I}] Suppose $\nu \in (0, 1]$ and $\|\nabla f(\bm x)\|_\infty \leq B$, the primal error of \textsc{DMO-FW-I} admits the following practical bound
\begin{equation}
h(\bm x_t) \leq \begin{cases}
\mathcal{O}\left( \frac{B C\sqrt{s}}{t^{\nu}}\right), & \|\nabla f(\bm x_t)\|_\infty \leq \frac{B}{t^\nu} \\
\mathcal{O}\left(\frac{B C\sqrt{s}(1-\delta)}{\delta}\right), & Otherwise.
\end{cases}
\label{thm:1:practical-rate}
\end{equation}
where $s \triangleq \max_{S\in \mathbb{M}} |S|$ is the maximal allowed sparsity.
\label{corollary:practical-bound}
\end{corollary}
\begin{remark}
 Our practical bound (\ref{thm:1:practical-rate}) reveals the essential advantage of FW-type algorithms: sparse solutions provide $\sqrt{s}$ better than $\sqrt{d}$. Fig. \ref{fig:least-square-k-support-norm} (left) also numerically substantiates our bound on the task of graph-structured linear regression. Fig. \ref{fig:least-square-k-support-norm} (right) illustrates the decay of $\|\nabla f(\bm x_t)\|_\infty$. Additionally, while the extra gradient decay assumption may seem unsatisfying, we observe it to be true in both our numerical simulations and real-world applications. In the other case, the gradient does not decay but assuming properly bounded; we have worst bound case $\mathcal{O}(BC\sqrt{s}(1-\delta)/\delta)$. Intuitively, the gradient decays when these NP-hard problems are ``easy''; otherwise, they are truly ``hard''. In the latter case, it is better to use the convergence rate over the relaxed ball. This assumption is therefore not needed in our set expansion scenario (Option II) below.
\end{remark} 

\begin{figure}
\centering
\includegraphics[width=.48\textwidth]{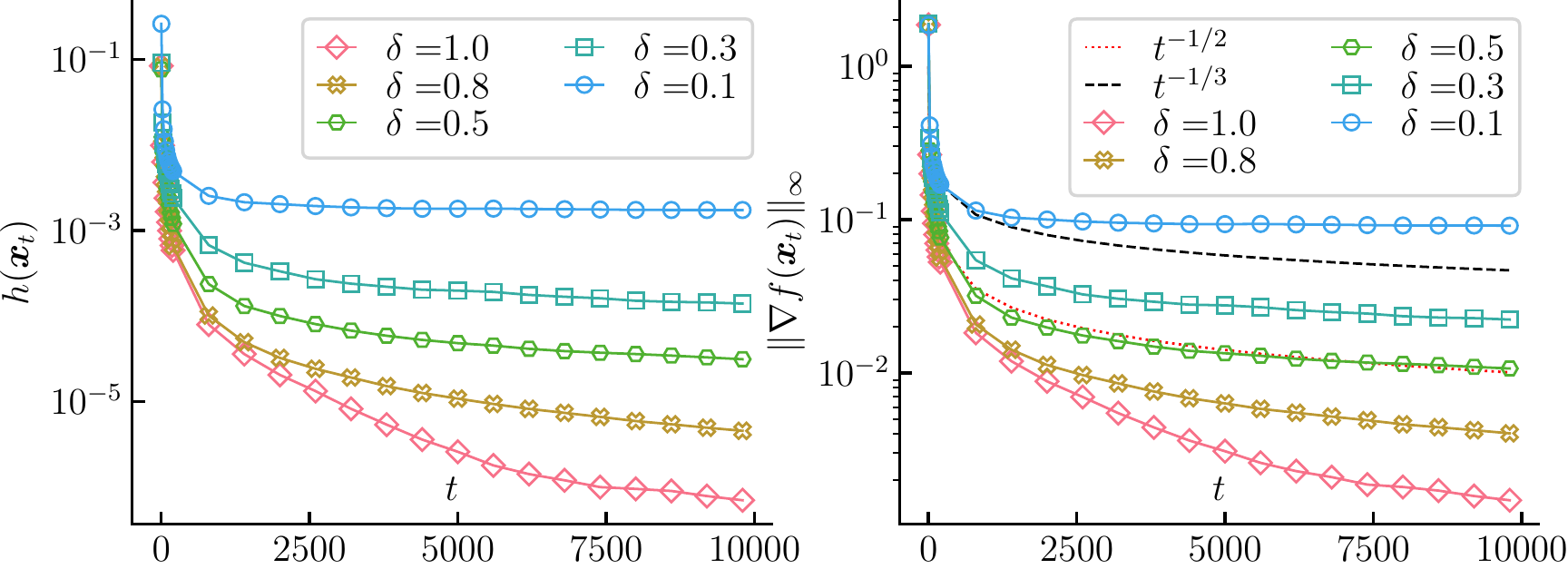}
\vspace{-5mm}
\caption{\textit{Left}: The primal error $h(\bm x_t)$ as a function of $t$ for \textsc{DMO-FW-I} with different $\delta$ on $s$-support norm ($s=5$) under the task of graph-structured linear sensing. \textit{Right}: The corresponding gradient norms as a function of $t$. More details of experimental setup can be found in Appendix \ref{appendix-k-support-norm}.\vspace{-5mm}}
\label{fig:least-square-k-support-norm}
\end{figure}

\paragraph{Expanding $\mathcal{D}$ to $\mathcal{D}/\delta$.} Even with the gradient decay assumption, the above analysis shows that even after infinite iterations, the approximation error could be large when $\delta$ is small. Alternatively, by allowing  $\bm x_t \in \mathcal{D}/\delta$, an expanded set of $\mathcal{D}$, we show that \textsc{DMO-FW-II} finds $\bm x_t$ in this \emph{relaxed} problem at rate $\mathcal{O}(1/t)$. To see this, note that $\bm x_{t+1}$ is a convex combination of approximate vectors $\tilde{\bm v}_i$. For each  $\tilde{\bm v}_t$, we enlarge its length to $\|\tilde{\bm v}_t\|/\delta$ so that $\langle \nabla f(\bm x_t), \tilde {\bm v}_t-\bm x_t\rangle$ is a lower bound of $\min_{\bm s \in \mathcal{D}/\delta} \langle \nabla f({\bm x}_t), \bm s - {\bm x}_t \rangle$. The rate of \textsc{DMO-FW-II} is stated in the following theorem.

\begin{theorem}[Convergence of \textsc{DMO-FW-II}]
\label{thm:2:fw-dmo-ii}
Under the same assumptions  as in Thm. \ref{thm:1:fw-dmo-i}, for any $t \geq 1$, the primal error of \textsc{DMO-FW-II} satisfies
\begin{equation}
h(\bm x_t) = f(\bm x_{t})-f(\bm x^*) \leq \frac{8 L C^2}{\delta^2 (t+2)}, \label{inequ:rate:thm2}
\end{equation}
where $\bm x_{t} \in \mathcal{D}/\delta$ and $\bm x^* \in \arg\min_{\bm x \in \mathcal{D}} f(\bm x)$.
\end{theorem}
\begin{remark}
The rate in (\ref{inequ:rate:thm2}) is comparable with that of the two variants in Matching Pursuit algorithms \citep{locatello2018matching},
which converge at $\mathcal{O}((\delta^2 (t+2))^{-1})$. To compare, the aforementioned variants must estimate parameter $L$, while \textsc{DMO-FW-II} $\delta$.
\end{remark}

\subsection{Convergence rate of \textsc{DMO-AccFW}}
Note that DMO-FW achieves $\mathcal{O}(1/t)$ convergence rate even when $\delta=1$. Inspired by \citet{garber21a}, we propose a variant FW-type method which uses an extreme point oracle, and can achieve a slightly better rate. The key idea  is that rather than computing the approximate DMO over $\bm z_t =  - \nabla f(\bm x_t)$, we instead compute the LMO over $\bm z_t =  - \left(\bm x_t - \tfrac{\nabla f(\bm x_t)}{L \eta_t}\right)$, which, if $\|\bm z_t\|_2$ is constant, is equivalent to finding the Euclidean projection of $\bm x_t - \tfrac{\nabla f(\bm x_t)}{L \eta_t}$ on $\mathcal D$, that is, it is a projected gradient method with the per-iteration cost of a FW method. This new variant is named \textsc{DMO-AccFW}.  With  proper assumptions, we prove that this variant converges at $\mathcal{O}(1/t^2)$. Before we state our main result, recall the quadratic growth condition.
\begin{definition}[$\mu$-Quadratic Growth Condition]
Let $f$ be continuous differentiable. $\mathcal{X}^* \triangleq \arg\min_{\bm x \in \mathcal{D}} f(\bm x)$. We say $f$ satisfies the quadratic growth condition on $\mathcal{D}$ if there exists a constant $\mu > 0$ such that
\begin{equation}
f(\bm x) - f(\bm x^*) \geq \tfrac{\mu}{2}
\| \bm x - [\bm x]_{\mathcal{X}^*}\|_2^2, \nonumber
\end{equation}
for all $\bm x \in \mathcal{D}$ where $[\bm x]_{\mathcal{X}^*} \triangleq \arg\min_{\bm z \in \mathcal{X}^*} \| \bm z - \bm x\|_2^2$.
\end{definition}

The quadratic growth condition is weaker than the strongly convex condition. For the remainder of this subsection we assume the convex function $f$ is $L$-smooth and satisfies a $\mu$-quadratic growth condition over $\mathcal D$ for option I, and over $\mathcal D/\delta$ for option II. We state our main result as follows:
\begin{theorem}
Assume further that $\bm x^*$ is on the boundary, i.e. $\bm x^* \in \mathcal{X}^*$ implies $\|\bm x^*\|_2 = C$. When $\delta=1$, for all $t \geq 1$, the primal error of \textsc{DMO-AccFW} satisfies
\begin{equation}
h(\bm x_{t}) \leq \frac{ 4 e^{4 L/\mu } h(\bm x_0)}{(t+2)^2}. \label{inequ:faster-rate}
\end{equation}
\end{theorem}

Standard FW method is well-known to have rate $\mathcal{O}(1/t^2)$ in general when the constraint set is a strongly convex set and $f$ is strongly convex \cite{garber2015faster}. Our key extension is to cases where $\mathcal D$ may not be strongly convex, as is often the case for obtaining sparsity. Note that quadratic growth condition is a slightly weaker condition than strongly convex; in particular, nonconvex functions may still have quadratic growth. When $\bm x^*$ is not on the boundary, we have the following general convergence rate.
\begin{corollary}
Denote $D_* =  \min_{\bm x^*\in\mathcal{X}^*}\| \bm x^*\|_2$. When $\delta=1$, for all $t \geq 1$, \textsc{DMO-AccFW-I} has the following convergence rate 
\begin{equation}
h(\bm x_{t}) \leq \min\left\{ \frac{3 L e^{2 L / \mu} (C^2 -D_*^2)}{t+2} + \frac{4 e^{4 L/\mu} h(\bm x_0)}{(t+2)^2}, Z_t\right\}, \nonumber
\end{equation}
where $Z_t = 2 L(5 C^2 - D_*^2) / (t+2)$.
\end{corollary}
Similar to the practical rate of \textsc{DMO-FW-I}, the practical rate and the rate on the expanded set $\mathcal{D}/\delta$ are stated in Thm. \ref{thm:5.9} and \ref{thm:5.10}.
\begin{theorem}[Practical rate of \textsc{DMO-AccFW-I}]
\label{thm:5.9}
When $\delta \in (0,1)$ and $\| \nabla f(\bm x)\|_\infty \leq B$ , then \textsc{DMO-AccFW-I} admits
\begin{equation}
h(\bm x_t) \leq \mathcal{O}\left( \sqrt{s}B C(1-\delta)/\delta \right). \nonumber
\end{equation}
Moreover, when $\| \nabla f(\bm x)\|_\infty \leq B/t^{\nu}$ with $\nu \in (0,1]$, we have
\begin{equation}
h(\bm x_t) \leq \mathcal{O}\left(\sqrt{s} B C/ t^\nu \right). \nonumber
\end{equation}
\end{theorem}
\begin{theorem}[Convergence of \textsc{DMO-AccFW-II}]
\label{thm:5.10}
When $\delta \in (0,1)$, then \textsc{DMO-AccFW-II} finds $\bm x_t \in \mathcal{D} /\delta$ and admits 
\begin{equation}
h(\bm x_t) \leq \frac{4 e^{4 L /\mu}}{(t+3)^2} h(\bm x_0) + \frac{28 L^2 (C^2/\delta^2 - D_*^2)}{ 5 \mu (t+3)}.
\end{equation}
\label{thm:4:acc-fw-dmo-ii}
\end{theorem}

The practical rate of \textssc{DMO-AccFW} shown in Thm \ref{thm:5.9} is the same as of \textsc{DMO-FW} as both two methods share the same DMO operator in worst case. For more comparisons between the above rates and the standard FW method, see Table \ref{tab:convergence-rate} in the Appendix \ref{appendix:tab:convergence-rate}.

\subsection{Case study: Graph-structured linear sensing}

The goal of graph-structured linear sensing is to recover a graph-sparse model $\tilde{\bm x}^*$ using fewer measurements than $d$. Specifically, measurements are generated as 
\[
\bm y = \langle \bm A, \tilde{\bm x}^* \rangle + \bm e, \qquad \bm e \sim \mathcal{N}(\bm 0, \sigma^2 \bm I_d).
\]
The sensing matrix $\bm A \in \mathbb{R}^{n\times d}$ is Gaussian random ($a_{i j } \sim \mathcal{N}(0, 1/\sqrt{n})$ independently). Our goal is to recover $\tilde {\bm x}^*$ using regression
$ f(\bm x) := \left\| \bm A \bm x - \bm y \right\|^2 / 2$ subject to graph-structured constraint $\mathcal{D}$. The following corollary provides the parameter estimation error bound for graph-structured linear sensing problem.

\begin{corollary}[High probability parameter estimation]
Let $h(\bm x_t)$ be primal error for \textsc{DMO-FW} or \textsc{DMO-AccFW}. The estimation error of the graph-structured linear sensing problem admits
\begin{equation}
\| \tilde{\bm x}^* - \bm x_t \|_2 \leq \sqrt{\frac{2 \sqrt{s}C \| \nabla f(\tilde{\bm x}^*)\|_\infty}{\mu}} + \sqrt{\frac{2 h(\bm x_t)}{\mu}}, \label{estimation-error}
\end{equation}
Moreover, with large enough $n$, there exists an universal constant $c$ such that with high probability:
\[
\| \tilde{\bm x}^* - \bm x_t \|_2 \leq \left( \sqrt{\frac{4 \sigma C \sqrt{s \log d / n}}{\frac{1}{2} - \frac{c s \log d}{n}}} + \sqrt{\frac{2 h(\bm x_t)}{\frac{1}{2} - \frac{c s \log d}{n}}} \right).
\]
\end{corollary}

\section{Empirical results}

We empirically evaluate our methods through the task of graph-structured linear regression problem over several graph-structured images. Our code and datasets will be made publically available upon publication, and is included in the submission. Our goal is to address two questions: \textbf{Q1.} Does \textsc{DMO-AccFW} speed up \textsc{DMO-FW}? \textbf{Q2.} How does the approximation quality and efficiency compare with baseline methods? \footnote{More experimental details are in Appendix \ref{appendix:section:experiments}.} Our code and datasets are accessible at \url{https://github.com/baojian/dmo-fw}.

\textit{DMO-AccFW converges faster than DMO-FW.} To answer \textbf{Q1}, our empirical results as illustrated in Fig. \ref{fig:test_mnist_dmo_5.0} on task of graph-structured sparse recovery clearly demonstrate that \textsc{DMO-AccFW} converges faster than \textsc{DMO-FW} and empirically matches the rate $\mathcal{O}(1/t^2)$ in (\ref{inequ:faster-rate}) and estimation error matches $\mathcal{O}(1/t)$ in (\ref{estimation-error}).

\textit{\textsc{DMO-AccFW} converges to a good local solution.} PGD-based methods and MP-based methods have been widely used on the task of sparse recovery. All methods shown linear convergence rate when training samples are sufficiently large. However, when the number of training samples are much less than $p$ (extremely challenging to recover the original signal), our method demonstrates sensing superiority. \begin{figure}[H]
\includegraphics[width=.5\textwidth]{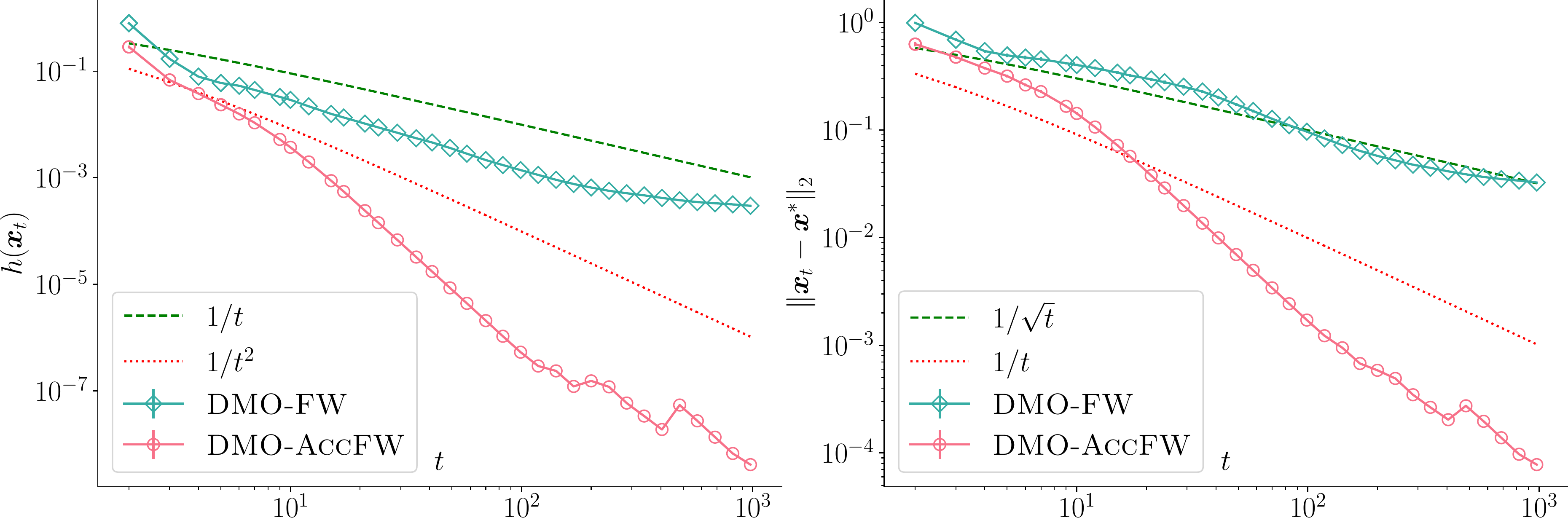}
\vspace{-5mm}
\caption{\textit{Left}: Primal error $h({\bm x}_t)$ as a function of $t$ on task of the graph-structured sparse recovery of MNIST[7]. \textit{Right}: The estimation error as a function of $t$. \vspace{-3mm}}
\label{fig:test_mnist_dmo_5.0}
\end{figure} Specifically, we pick $n=2.5 \cdot |\operatorname{supp}(\bm x^*)|$ samples. We run each experiment for 20 trials, and compare our methods against the generalized MP (\textsc{Gen-MP}) discussed in \citet{locatello2018matching} where each constant curvature is estimated by the maximal eigenvalue of $\bm A^\top \bm A$, \textsc{CoSAMP} \cite{needell2009cosamp},  \textssc{GraphCoSAMP} \cite{hegde2015nearly}, and \textsc{Graph-IHT} \cite{hegde2016fast} (the PGD-based method). The DMO we used is the head projection of \citet{hegde2015nearly}. We use Option \textsc{I} for \textsc{DMO-AccFW} and simply set $L=1$.

\begin{figure}[H]
\includegraphics[width=.48\textwidth]{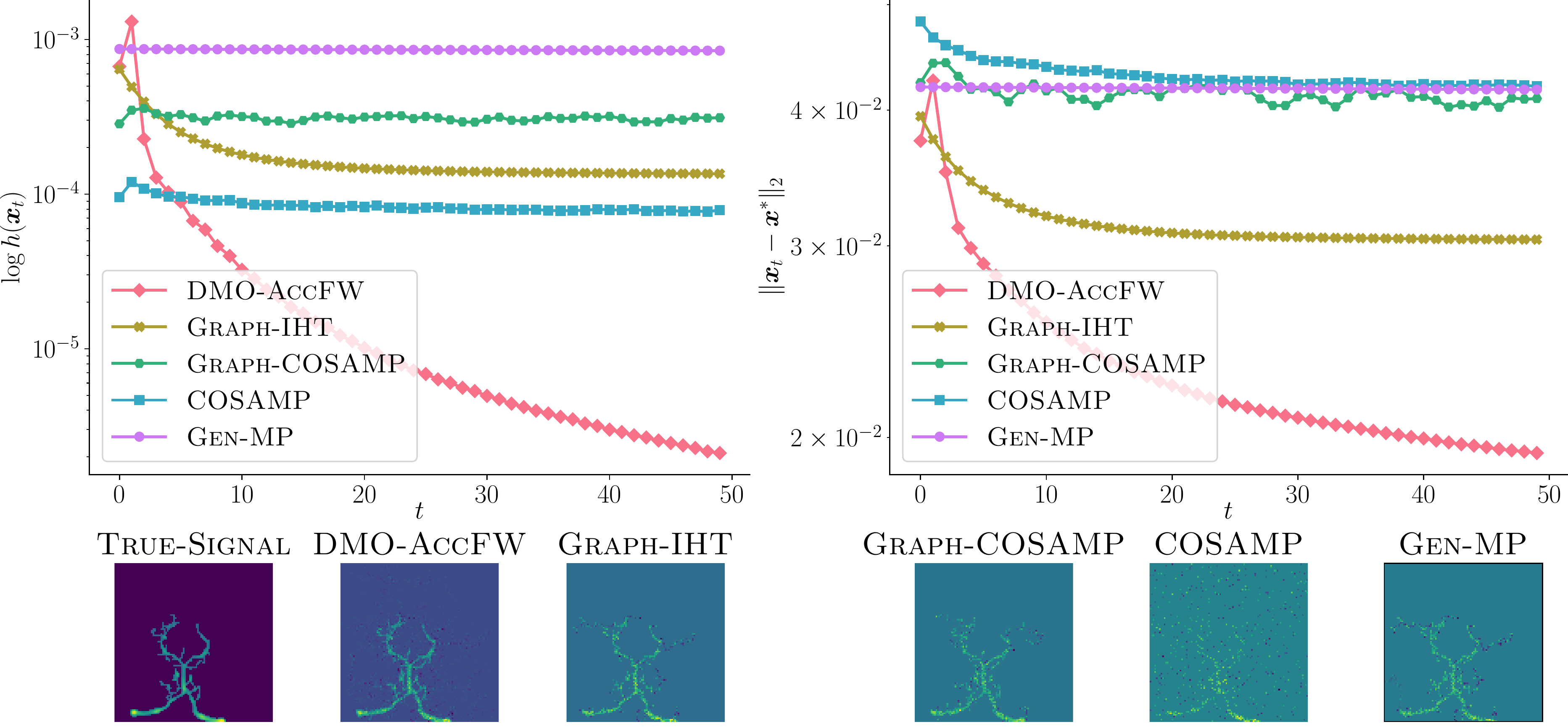}

\caption{The performance of methods on the task of graph-structured linear regression task. \textit{Top}: 
\textsc{DMO-AccFW} vs other baseline methods on primal error $ h(\bm x_t)$ (\textit{left}) and variable suboptimality $\|\bm x_t-\bm x^*\|_2$ (\textit{right}) as a function of $t$. 
\textit{Bottom}:  
Recovered sparse images $\bm x_t$ vs truth image (bottom right) after $50$ iterations. }
\label{fig:sparse-image-recovery}
\end{figure} 
 
The convergence comparison is shown in Fig. \ref{fig:sparse-image-recovery}. Compared with the PGD-based method \textsc{GraphIHT}, \textsc{DMO-AccFW} converges faster to a good local optimal while \textsc{GraphIHT} appears to be stuck in a local minimum. Three MP-based methods also failed to cleanly recover the original signal. Our \textsc{DMO-AccFW} is the fastest one to converge to a good local minimum at very early stage. The bottom recovered image also demonstrate that \textsc{DMO-AccFW} returns sparse solution at early stage. The efficiency of \textsc{DMO-AccFW} has been demonstrated in Table \ref{experiments:run-time} where it is several times faster than baselines because baselines either have multiple projections or need to have minimization step per-iteration.

\begin{center}
\vspace{-5mm}
\begin{table}[!th]
\caption{The comparison of run time of all methods.}
\centering
\begin{tabular}{p{0.15\textwidth}|p{0.2\textwidth}} \toprule
Method & Run time (seconds) \\\hline 
\textsc{GraphCoSAMP} & 662.10$\pm$252.21 \\\hline 
\textsc{CoSAMP} & 531.57$\pm$45.47 \\\hline 
\textsc{GenMP} & 661.68$\pm$177.65 \\\hline 
\textsc{GraphIHT} & 445.89$\pm$22.63 \\\hline 
\textsc{DMO-AccFW} & \textbf{87.61$\pm$23.39} \\\bottomrule
\end{tabular}
\label{experiments:run-time}
\end{table}
\end{center}

\section{Discussion and Conclusion}

We have studied the approximate FW-type methods for GSCO problems. We first demonstrate that there exist adversarial examples such that two popular inexact LMOs are at least as hard to compute as the exact LMO. Instead, we consider an inexact DMO which is equivalent to an approximation on the inner product rather than the gap, and prove that the inexact DMO is equivalent to the inexact IPO. The standard FW admits $\mathcal{O}((1-\delta)\sqrt{s}/\delta)$ and our accelerated version has rate $\mathcal{O}(1/t^2)$ with a proper condition. We also prove that a relaxed version of FW admits $\mathcal{O}(L C^2 /(\delta^2 t))$ for general convex functions where the iterates $\bm x_t$ may be infeasible, and converge to a relaxed set.

One weakness of the present work is that our established bound depends on the decay of the gradient with the fixed step size strategy, which is observed in practice but difficult to show theoretically. If $\bm x^*$ is on the boundary, then $\|\nabla f(\bm x^*)\|_\infty$ may never decay to 0, and the proposed methods converge to an error neighborhood. We also do not consider line search, away steps, or fully-corrective versions, and limit our attention to the traditional fixed step size strategies. Finally, our initial experiments 
indicate that inexact FW-type methods are attractive for this type of problems and one is encouraged to find faster methods based on fully-corrective or other methods.

\section{Acknowledgement}

Both authors thank the anonymous reviewers for their helpful comments. We would like to thank Steven Skiena for his support during the start of this project. The work of Baojian Zhou is partially supported by the startup funding from Fudan University.

\bibliography{references}
\bibliographystyle{icml2022}

\newpage
\appendix
\onecolumn

Section \ref{appendix:proofs} provides all missing proofs. Section \ref{appendix:adv-examples} presents the adversarial examples at non-optimal points $\bm x_t$. Section \ref{appendix-k-support-norm} gives the experimental details of Fig. \ref{fig:least-square-k-support-norm}. Section \ref{appendix:section:experiments} provides more experimental results on graph-structured sparse recovery problem. Finally, in Section \ref{appendix:section:dmo}, we present two DMOs for $\mathbb{M}(\mathbb{G}, s, g)$ and briefly discuss other DMOs. Our code, datasets, and results are also provided in the supplementary material and will be made available on publish.

\section{Proofs}
\label{appendix:proofs}

\subsection{Proof of Theorem \ref{thm:np-hard}}

Before proving the NP-hardness result, we first introduce \textit{the group model selection (GMS) problem} \cite{baldassarre2016group}. Notations of this problem are:
Let $\tau(\bm y)$ be the indicator vector where
\[
\tau(\bm y)_i = \begin{cases} 1 & \text{if } y_i \ne 0,\\ 
0 & \text{otherwise.}
\end{cases}
\]

Given the ground set $[N]$, the group structure $\mathbb{B} \triangleq \{B_1, B_2,\ldots, B_T\}$ is a collection of index sets with
\[
B_i \subseteq [N], \quad |B_i| \leq b_i, \quad \cup_{B_i \in \mathbb{B}} B_i = [N].
\]
The group indicator vector $\bm \omega$ denotes activity of elements in $\mathbb{B}$, that is, $\omega_i = 1$ if $B_i$ is active 0 otherwise. To encode the group structure $\mathbb{B}$, an indicator matrix $\bm A \in \mathbb{B}^{N\times T}$ is defined where each entry $a_{i j} = 1$ if $i \in B_j$ and 0 otherwise. Based on this definition, $\bm A \bm \omega \geq \tau(\bm y)$ means that for each nonzero $y_i$, there at least one active group in $\mathbb{B}$ covers $y_i$. Given any $\bm y\in \mathbb{R}^n$, a best $\gamma$-group sparse approximation $\hat{\bm y}$ is given by
\begin{equation}
\hat{\bm y} \in \argmin{\bm z \in \mathbb{R}^N} \left\{ \| \bm y - \bm z\|_2^2 : \| \bm z\|_{ \mathbb{B},0} \leq \gamma \right\}, \text{ where } \| \bm z\|_{\mathbb{B}, 0} \triangleq \min_{\bm \omega \in \mathbb{B}^T} \left\{ \sum_{j=1}^T \omega_j : \bm A \bm \omega \geq \tau(\bm z) \right\},
\end{equation}
where $\|\bm z \|_{ \mathbb{B},0}$ expresses the minimal number of active groups that cover $\bm z$.

\begin{problem}[The group model selection (GMS) problem \cite{baldassarre2016group}]
Let $[N]$ be the ground set. Assume the group structure $\mathbb{B} := \{B_1,B_2,\ldots, B_T\}$ where each $B_i \subseteq [N]$ and $\cup_{B_i \in \mathbb{B}} B_i = [N]$.
Given any input $\bm y \in \mathbb{R}^N$, a $\gamma$-group cover for its $\gamma$-group sparse approximation is expressed as follows
\begin{equation}
\mathcal{S}(\gamma,\hat{\bm y}) \in \argmax{\mathcal{S} \subseteq \mathbb{B}} \left\{ \| \bm y_I \|_2^2 : I = \cup_{S_i \in \mathcal{S}} S_i, |\mathcal{S}| \leq \gamma \right\}, \label{equ:18}
\end{equation}
where the $\gamma$-group cover $\mathcal{S}(\gamma, \hat{\bm y})$ is a group cover for $\hat{\bm y}$ with at most $\gamma$ groups, that is
\begin{equation}
\mathcal{S}(\gamma, \hat{\bm y}) = \left\{ B_i \in \mathbb{B} : \bm \omega \in  \mathbb{B}^{T}, \omega_i = 1, \bm A \bm \omega \geq \tau(\hat{\bm y}), \sum_{j=1}^T \omega_j \leq \gamma \right\}.
\end{equation}
The GMS problem is to find such cover $\mathcal{S}(\gamma,\hat{\bm y})$ for any given $\bm y$ and $\gamma$.
\end{problem}

\begin{lemma}[The NP-hardness of the GMS problem \cite{baldassarre2016group}]  Given the ground set $[N]$ and $\gamma$ is a positive integer. Suppose $\bm y \in \mathbb{R}^N$ and $\mathbb{B} \in \mathcal{P}([N])$. The group model selection problem is NP-hard.
\end{lemma}

To prove the NP-hardness of graph-structured LMO, we make a reduction from the GMS problem to it. The main observation is that, given any instance of GMS problem, one can construct $\mathbb{M}$ such that $\mathcal{S}(\gamma,\hat{\bm y})$ can be recovered from $\operatorname{supp}(\bm v_t)$.

\begin{reptheorem}{thm:np-hard}[Hardness of graph-structured LMO]
Given $C >0, a = 0, b = 1$, convex differentiable  $f:\mathbb{R}^d\rightarrow \mathbb{R}$, and the graph-structured model $\mathbb{M}$, the computation of the graph-structured LMO defined in (\ref{equ:lmo}) is NP-hard. 
\end{reptheorem}
\begin{proof}

\begin{table}[!ht]
\centering
\caption{The comparison of parameter configurations between two problems.}
\begin{tabular}{p{0.15\textwidth}|p{0.3\textwidth}|p{0.4\textwidth}}
\toprule
  & The group model selection problem & The computation of graph-structured LMO \\
\hline 
Ground set & $[N] \triangleq \{1, 2, \ldots, N\}$ & $[d] \triangleq \{1, 2, \ldots, d\}$ \\
\hline 
Input vector & $\bm y$ & $a \bm x_t + b \nabla f(\bm x_t)$, $a, b \in \mathbb{R}^d$ and $\nabla f(\bm x_t) \in \mathbb{R}^d$ \\
\hline 
Structure model & $\mathbb{B} = \{B_1, B_2, \ldots, B_T\}$ & $\mathbb{M} = \{ S_1, S_2, \ldots, S_{M}\}$ \\
\hline 
Group size & $\gamma$ & $g$ \\
\hline 
Solution & $\hat{\bm y}_I = \bm y_I$ where $I \subseteq \mathbb{B}$ with $|I| = \gamma$ & $\bm v_t = \frac{-C\cdot \nabla f(\bm x_t)_{S^*}}{\| \nabla f(\bm x_t)_{S^*}\|_2}$ where $S^* \in \mathbb{M}$\\\bottomrule
\end{tabular}
\label{tab:model-reduction}
\end{table}	
For any given instance of the GMS problem defined in Table \ref{tab:model-reduction}, one can find an instance of the graph structured LMO problem: Let $d = N$, fix $a=0, b= 1$ and let the convex differentiable function $f(\bm x) = \bm x^\top \bm y$, so that $\nabla f(\bm x_t) = \bm y$. The rest is to construct $\mathbb{M}$ so that $\operatorname{supp}(\bm v_t)$ corresponds to $I$. Follow a similar argument presented in Sec. \ref{section:hardness-result}, we have
\begin{align*}
\min_{\bm s \in \mathcal{D}(C, \mathbb{M})} \left\langle \nabla f(\bm x_t), \bm s \right\rangle &=  \min_{\bm s \in \operatorname{conv}\{\cup_{S \in \mathbb{M}} \mathcal{B}(S, C)\}} \left\langle \bm y, \bm s \right\rangle \\
&=  \min_{\bm s \in \cup_{S \in \mathbb{M}} \mathcal{B}(S, C)} \left\langle \bm y, \bm s \right\rangle \\
&= \min_{\bm s \in \mathcal{B}(S^*, C)} \left\langle \bm y_{S^*}, \bm s  \right\rangle, \nonumber
\end{align*}
where $S^*$ is an optimal support set and $\mathcal{B}(S,C) = \{ \operatorname{supp}(\bm x) = S, \|\bm x\|_2 \leq C \}$. Therefore, we have the following equivilent representation of $\bm v_t$
\begin{equation}
\bm v_t = - \frac{C \cdot \bm y_{S^*}}{\| \bm y_{S^*}\|_2}, S^* \in \argmax{S \in \mathbb{M}} \| \bm y_S \|_2^2, \nonumber
\end{equation}
where we let $C=1$. Now let $\mathbb{M} = \{S_i : S_i = \cup_{Q_j \in \mathcal{S}} Q_j, |\mathcal{S}| \leq \gamma , \mathcal{S} \subseteq \mathbb{B}\}$. That is, $\mathbb{M}$ is problem space of (\ref{equ:18}). Notice further that, given the underlying graph $\mathbb{G}(\mathbb{V},\mathbb{E})$, one can define $\mathbb{M}(s,g=\gamma) \triangleq \mathbb{M}$ with $s=\max_i |S_i|$.
Therefore, with this specific configuration, the solution $S^*$ of the graph-structured LMO is $I$. One can immediately recover $\mathcal{S}(\gamma,\hat{\bm y})$ from $I$ by using the fact that $I = S^* = \cup_{Q_j \in \mathcal{S}} Q_j$ for a specific $\mathcal{S}$.
\end{proof}

\subsection{Proof of Theorem \ref{theorem:dmo-to-ipo}}
\begin{reptheorem}{theorem:dmo-to-ipo}
Given the set $\mathcal{D}$ and suppose $S \in (\delta,\bm z, \mathcal{D})$-DMO. Define the approximate supporting vector $\tilde{\bm v}_t \triangleq -C\cdot\bm z_S / \|\bm z_S\|_2$.  Then, $\tilde{\bm v}_t \in (\delta, \bm z, \mathcal{D})$-IPO.
\end{reptheorem}
\begin{proof}
Since $S$ is a support returned by $(\delta, \bm z, \mathcal{D})$-DMO and $\tilde{\bm v}_t = \frac{- C}{\| \bm z_S\|_2}\bm z_S$, we have the following
\begin{align}
\left\langle \bm z, \tilde{\bm v}_t \right\rangle &= \left\langle \bm z, \frac{-C}{\|{\bm z}_S\|_2} \bm z_S \right\rangle \nonumber\\
&= - C \cdot \|\bm z_S\|_2 \nonumber\\
&\leq - C \cdot \delta \max_{S^\prime \in \mathbb{M}} \| \bm z_{S^\prime}\|_2 \nonumber\\
&= \delta \cdot \min_{S^\prime \in \mathbb{M}} \| \bm z_{S^\prime}\|_2\cdot (-C), \label{thm:4.3-01}
\end{align}
where the last inequality due to the fact that $S$ is in $(\delta,\bm z, \mathcal{D})\operatorname{-DMO}$. As defined in (\ref{def:lmo}), in the rest, we shall prove that $\left\langle \bm z, \tilde{\bm v}_t \right\rangle \leq \delta \cdot \min_{\bm s \in \mathcal{D}} \left\langle \bm z, \bm s \right\rangle$. Recall
 $\mathcal{D} = \operatorname{conv}(\cup_{I \in \mathbb{M}} \{\bm w \in \mathbb{R}^d: \operatorname{supp}(\bm w) \subseteq I, \| \bm w\|_2 \leq C \})$. Denote the ball induced by $I$ as $\mathcal{B}(I) := \{ \bm x\in \mathbb{R}^d: \operatorname{supp}(\bm x) \subseteq I, \| \bm w\|_2 \leq C\}$. Notice that $\bm s$ be must in the set of extreme points.  We have
\begin{align*}
\min_{\bm s \in \mathcal{D}} \left\langle \bm z, \bm s \right\rangle &= \min_{\bm s \in \operatorname{conv}(\cup_{I \in \mathbb{M}} \mathcal{B}(I))} \left\langle \bm z, \bm s \right\rangle \\
&= \min_{\bm s \in \cup_{I \in \mathbb{M}} \mathcal{B}(I)} \left\langle \bm z, \bm s \right\rangle \\
&= \min_{\bm s \in \mathcal{B}(I^*)} \left\langle \bm z_{I^*}, \bm s  \right\rangle, \nonumber
\end{align*}
where we denote $I^*$ as the support of $\bm s$ in $\mathbb{M}$. By Cauchy-Schwarz inequality, we always have $- \|\bm z_{I^*}\|_2 \cdot\|\bm s\|_2 \leq \left\langle \bm z_{I^*}, \bm s\right\rangle$. When $\bm s = - C\cdot\bm z_{I^*} / \| \bm z_{I^*}\|_2$, it attains the minimal value $-C \cdot \| \bm z_{I^*}\|_2$. Therefore, we have
\begin{align*}
\min_{\bm s\in\mathcal{D}} \langle \bm z, \bm s\rangle = \| \bm z_{I^*}\|_2 \cdot (- C) \text{ for some } I^* \in \mathbb{M}.
\end{align*}
As $I^* \in \mathbb{M}$, we continue (\ref{thm:4.3-01}) to have
\begin{align*}
\langle \bm z, \tilde{\bm v}_t \rangle &\leq \delta \cdot \min_{S^\prime \in \mathbb{M}} \| \bm z_{S^\prime}\|_2 (-C) \\
&\leq \delta \cdot \| \bm z_{I^*}\|_2 \cdot (- C) \\
&= \delta \cdot \min_{\bm s \in \mathcal{D}} \langle \bm z, \bm s\rangle.
\end{align*}
The above inequality indicates that given $S$ that satisfies DMO property, then $\tilde{\bm v}_t$ defined based on $S$ is a solution of IPO operator. We prove the lemma.
\end{proof}
\begin{remark}
Theorem \ref{theorem:dmo-to-ipo} provides an easier way to find a solution of IPO. It also indicates IPO operator and DMO operator are equivalent. That is, for an existing $\bm v_t \in (\delta, \bm z, \mathcal{D})$-IPO, one can find $\operatorname{supp}(\bm v_t) \in (\delta, \bm z, \mathcal{D})$. In case of $k$-support norm, the proof of exact equivalence (i.e., $\delta=1$) appears in the Proposition 2 of \citet{mcdonald2016fitting}, Section 2.1 of \citet{argyriou2012sparse},  and Lemma 2 of \citet{jacob2009group} where calculating the dual norm is equivalent to solving LMO. In particular, let $I_k$ be the top-$k$ largest magnitudes of $\bm u$, then any $\bm s$ such that $\langle \bm u, \bm s \rangle = \|\bm u_{I_k}\|$ will be a solution of LMO, with 
$\bm s =- \bm u_{I_k} / \|\bm u_{I_k}\|_2$.
\end{remark}

\subsection{Proof of Theorem \ref{thm:1:fw-dmo-i}}

Denote the primal error $h(\bm x_{t}) := f(\bm x_{t}) - f(\bm x^*)$ and assume the step size $\eta_t = 2/(t+2)$. Note that the solution $\bm x_{\bar{t}}$ always has $h(\bm x_{\bar{t}}) \leq h(\bm x_t)$. We assume $f$ is $L$-smooth and denote
$\bm x^* \in \arg\min_{\bm x\in \mathcal{D}} f(\bm x)$ as a minimizer of (\ref{def:problem}). We begin to introduce a key lemma as the following

\begin{lemma}
Given $h:\mathbb{R}^d\rightarrow \mathbb{R}^+$ and the following first-order non-homogeneous recurrence relation
\begin{equation}
h(\bm x_{t+1}) \leq \left(1-\frac{2\delta}{t+2}\right) h(\bm x_{t}) + \frac{A_t}{(t+2)^2}, \label{inequ:non-homo-recurr}
\end{equation}
where $\delta \in (0,1]$, $t \geq 0$, and $A_t \in \mathbb{R}^+$. Then, $\forall t \geq 1$, we have
\begin{equation}
h(\bm x_{t}) \leq \begin{dcases*}
\frac{(1-\delta)9^\delta}{(t+2)^{2\delta}} \cdot h(\bm x_0) + \frac{\ln(t+1) + 1}{(t+2)^{2\delta}} \cdot A  & if $\delta \in (0,\frac{1}{2}]$ \\
\min \left\{ \frac{(1-\delta)9^\delta}{(t+2)^{2\delta}} \cdot h(\bm x_0) + \frac{\ln(t+1) + 1}{t+2} \cdot A,  \frac{3(1-\delta)h(\bm x_0) + A}{(2\delta-1)(t+2)}\right\}  & if $\delta \in (\tfrac{1}{2},1]$,
\end{dcases*}
\end{equation}
where $A$ be such that $A_t \leq A$.
\label{lemma:currence-inequality}
\end{lemma}
\begin{proof}
We directly expand the recurrence relation (\ref{inequ:non-homo-recurr}) as the following
\begin{align}
h(\bm x_{t+1}) &\leq \left(1-\frac{2\delta}{t+2}\right) h(\bm x_{t}) + \frac{A_t}{(t+2)^2} \nonumber\\
&\leq \left(1-\frac{2\delta}{t+2}\right) \left(1-\frac{2\delta}{t+1}\right) h(\bm x_{t-1}) + \left(1-\frac{2\delta}{t+2}\right)\left(\frac{A_{t-1}}{(t+1)^2}\right) +  \frac{A_t}{(t+2)^2} \nonumber\\
&\vdots  \nonumber\\
&\leq \prod_{i=0}^t\left(1-\frac{2\delta}{i+2}\right)\cdot h(\bm x_0) + \sum_{j=0}^{t-1}\left\{ \frac{A_j}{(j+2)^2} \prod_{k=j+1}^{t}\left(1-\frac{2\delta}{k+2}\right) \right\} + \frac{A_t}{(t+2)^2} \triangleq H(t,\delta), \label{inequ:lemma-a6-1}
\end{align}
where we define the above tight bound (\ref{inequ:lemma-a6-1}) as $H(t,\delta)$. In the rest, we aim at getting an explainable upper bound of $H(t,\delta)$. First, note that
\begin{equation}
\prod_{i=j}^t\left(1-\frac{2\delta}{i+2}\right) 
\overset{(a)}{\leq}  \left(e^{-2\delta}\right)^{ \sum_{i=j}^t\frac{1}{i+2}} \overset{(b)}{\leq} \left(\frac{ j+2}{t+3}\right)^{2\delta}.
\label{inequ:lemma-a6-3}
\end{equation}
where (a) is from $1-x\leq e^{-x}$ and (b) is from the integral bound:
\[
\sum_{i=j}^t \frac{1}{i+2} \geq \int_{j}^{t+1} \frac{1}{\tau+2}d\tau = \ln(t+3)-\ln(j+2).
\]
Hence, taking 
\[
\prod_{i=0}^t\left(1-\frac{2\delta}{i+2}\right) = (1-\delta)\prod_{i=1}^t\left(1-\frac{2\delta}{i+2}\right),
\]
then $H(t,\delta)$ reduces to
\begin{eqnarray}
H(t,\delta) &\leq& (1-\delta)\left(\frac{3}{t+3}\right)^{2\delta} \cdot h(\bm x_0) + \sum_{j=0}^{t-1}\left\{ \frac{A_j}{(j+2)^2} \left(\frac{j+3}{t+3}\right)^{2\delta} \right\} + \frac{A_t}{(t+2)^2} \nonumber\\
&=& (1-\delta)\left(\frac{3}{t+3}\right)^{2\delta} \cdot h(\bm x_0) + \underbrace{\left(\frac{1}{(t+3)^{2\delta}}\sum_{j=0}^{t}\frac{A_j (j+3)^{2\delta}}{(j+2)^2}\right)}_{(\star)} \label{equ:recurrence-inequality}
\end{eqnarray}
Note that $\sum_{j=0}^{t} \tfrac{(j+3)^{2\delta}}{(j+2)^2}$ is increasing with respect to $\delta$.
Therefore, for $\delta \leq  1/2$, 
\begin{equation}
(\star) \overset{\delta\leq\tfrac{1}{2}}{\leq } \frac{1}{(t+3)^{2\delta}} \sum_{j=0}^t \underbrace{\left( \frac{1}{j+2} + \frac{1}{(j+2)^2} \right)}_{=\frac{j+3}{(j+2)^2}} \overset{(c)}{\leq} \frac{1}{(t+3)^{2\delta}} \left( \ln(t+2) + 1 - \frac{1}{t+2} \right)
\leq \frac{\ln(t+2) + 1}{(t+3)^{2\delta}},
\end{equation}
where (c) is again because of the integral bound, 
\begin{equation}
\sum_{j=0}^t \left(\frac{1}{j+2}+\frac{1}{(j+2)^2}\right) \leq \int_{-1}^{t} \left(\frac{1}{j+2}+\frac{1}{(j+2)^2}\right) dj = \ln(t+2)-0 - \frac{1}{t+2} + 1.
\label{eq:helper1}
\end{equation}

On the other hand,  $\sum_{j=0}^{t} \left(\tfrac{j+3}{t+2}\right)^{2\delta}$ is decreasing with respect to $\delta$.
Therefore, for $\delta \geq  1/2$, by similar logic,
\begin{equation}
(\star) \overset{\delta\geq\tfrac{1}{2}}{\leq } \frac{1}{t+3} \sum_{j=0}^t \left( \frac{1}{j+2} + \frac{1}{(j+2)^2} \right) 
\leq \frac{\ln(t+2) + 1}{t+3}.
\end{equation}
Overall, this gives us
\begin{equation}
h(\bm x_{t+1}) \leq P(t+1,\delta)  \triangleq \begin{cases}
(1-\delta)\left(\frac{ 3}{t+3}\right)^{2\delta} \cdot h(\bm x_0) + \frac{\ln(t+2)+1}{(t+3)^{2\delta}}\cdot a & \text{ if } \delta \in (0,1/2]\\
(1-\delta)\left(\frac{ 3}{t+3}\right)^{2\delta} \cdot h(\bm x_0) + \frac{\ln(t+2)+1}{t+3}\cdot a & \text{ if } \delta \in (1/2,1).
\end{cases} \label{inequ:lemma:a4:1}
\end{equation}
However, for the case of $\delta \geq 1/2$, this is not the best we can do. If we include the $\delta$ in the error term, we can achieve an $O(A_t/t)$ error rate. In particular, for $\delta \in (1/2,1]$ and $t \geq 1$, we have the following
\begin{equation}
h(\bm x_{t}) \leq S(t,\delta) \triangleq \frac{3(1-\delta)h(\bm x_0) + A_t}{(2\delta-1)(t+2)}.
\label{inequ:new-recurrence}
\end{equation}
We prove (\ref{inequ:new-recurrence}) by the induction. For $t = 0$, the initial recurrence relation can be expressed as
\begin{align*}
h(\bm x_1) \leq (1-\delta) h(\bm x_0) + \frac{A_0}{4} &\leq \frac{3}{3}\cdot \frac{(1-\delta)h(\bm x_0)}{(2\delta-1)} + \frac{A_0}{(2\delta-1)(3)} \\
&\leq \frac{3(1-\delta)h(\bm x_0) + A_0}{(2\delta-1)(t+2)}.
\end{align*}
Suppose $t = k$, (\ref{inequ:new-recurrence}) is true, and we consider $t = k+1$. Then, defining $B_t = \frac{3(1-\delta)h(\bm x_0)+A}{2\delta-1}$,
\begin{align*}
h(\bm x_{k+1}) &\leq \left( 1 - \frac{2\delta}{k+2} \right) h(\bm x_k) + \frac{A_k}{(k+2)^2} \\
&\leq \left( 1 - \frac{2\delta}{k+2} \right) \frac{B_k}{k + 2} + \frac{A_t}{(k+2)^2} \\
&= \frac{A}{k+3} + \left(\frac{A}{k+2}-\frac{A}{k+3}\right) -\frac{2\delta A}{(k+2)^2} + \frac{a}{(k+2)^2}\\
&= \frac{A}{k + 3} + \underbrace{A\left( \frac{1}{(k+2)(k+3)} - \frac{2\delta}{(k+2)^2}\right) + \frac{a}{(k+2)^2}}_{B},
\end{align*}
where 
\[
(2\delta-1)(k+2)^2(k+3) B =
3(\delta-1)(2\delta-1)h(\bm x_0) k + 6(1-3\delta)(1-\delta)h(\bm x_0) - a \leq 0
\]
taking $C = 3(1-\delta)h(x_0)$ and thus $A = \frac{C}{2\delta-1}+\frac{a}{2\delta-1}$, then
\begin{eqnarray*}
(2\delta-1)(k+2)^2(k+3) B 
&=&(2\delta-1)A((k+2)-2\delta (k+3)) + (2\delta-1)a(k+3)\\
&=& C\underbrace{((1-2\delta)(k+3)-1)}_{\leq 0 \text{ if } \delta \geq 1/2}\underbrace{-a((2\delta-1)(k+3)+1)+(2\delta-1)a(k+3)}_{=-a}\\
&\leq& 0
\end{eqnarray*}
provided $\delta \in (1/2,1]$. Hence, we have the inequality (\ref{inequ:new-recurrence}).

Overall, this gives, for $t \geq 1$,
\begin{equation}
h(\bm x_{t}) \leq  \hat P(t,\delta) \triangleq \begin{dcases*}
P(t,\delta)  & if $\delta \in (0,\frac{1}{2}]$ \\
\min \left\{ P(t,\delta), S(t,\delta)\right\}  & if $\delta \in (\tfrac{1}{2},1]$ 
\end{dcases*} \label{inequ:lemma:a4:2}
\end{equation}
which concludes the proof.
\end{proof}

\begin{figure}[H]
\centering
\includegraphics[width=.95\textwidth]{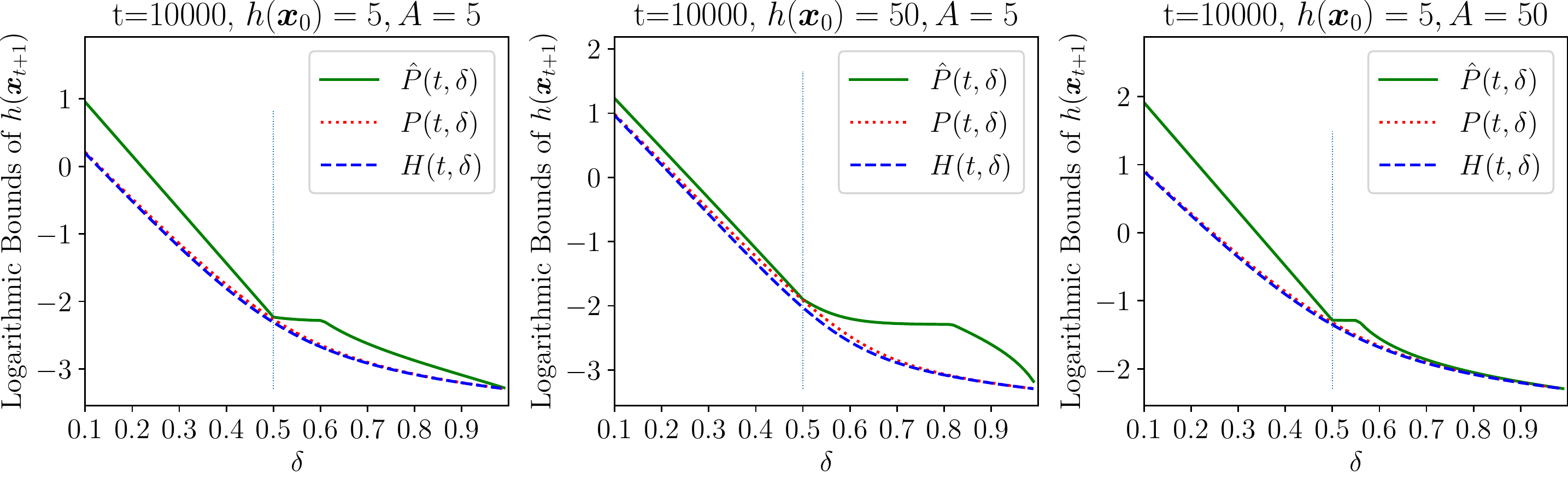}
\caption{The comparison of three different upper bounds of $h(\bm x_{t+1})$ in Lemma \ref{lemma:currence-inequality}. $H(t,\delta)$ is defined in (\ref{inequ:lemma-a6-1}), $P(t,\delta)$ is defined in (\ref{inequ:lemma:a4:1}), and $\hat{P}(t,\delta)$ is defined in (\ref{inequ:lemma:a4:2}).}
\label{fig:compare-upper-bound}
\end{figure}

Figure \ref{fig:compare-upper-bound} compares the rates $P(t,\delta)$ with $\hat P (t,\delta)$ in the regime of $\delta \geq 1/2$, as a way of bounding the error term $H(t,\delta)$. We are ready to prove our universal rate of \textsc{DMO-FW-I}.

\begin{reptheorem}{thm:1:fw-dmo-i}[Universal rate of \textsc{DMO-FW-I}] For all $t \geq 1$, the primal error of \textsc{DMO-FW-I} satisfies
\[
h(\bm x_t) \leq \begin{cases}
\min\left\{ 2 C\sqrt{s} \|\nabla f(\bm x_t)\|_\infty, P(\delta,2\delta) \right\} & \delta \in (0,\tfrac{1}{2}] \\
\min \left\{ P(\delta, 1),  \tfrac{3(1-\delta)h(\bm x_0) + A_t}{(2\delta-1)(t+2)}\right\} & \delta \in (\tfrac{1}{2},1]
\end{cases}
\]
where $s$ is maximal allowed sparsity, $P(\delta,\nu)$ is
\begin{equation}
P(\delta,\nu)\triangleq\frac{(1-\delta)9^\delta}{(t+2)^{2\delta}} \cdot h(\bm x_0) + \frac{\ln(t+1) + 1}{(t+2)^{\nu}} \cdot A_t \nonumber
\end{equation}
and $A_t \triangleq \max_{i \in [t]} 4(1-\delta)(i+2) |\left\langle -\bm x_i, \nabla f(\bm x_i) \right\rangle| + 8 L C^2$.
\end{reptheorem}

\begin{proof}

When DMO provides an exact solution, i.e., $\delta=1$, then for all $t \geq 0$, $h(\bm x_{t+1})$ of \textsc{DMO-FW-I} is recursively bounded as 
\[
h({\bm x_{t+1}}) \leq \left(1-\frac{2}{t+2}\right)h({\bm x_t}) +  \frac{2 L (2 C)^2 }{(t+2)^2},
\]
which eventually leads to $h(\bm x_{t}) \leq 8 L C^2 / (t+2)$ for all $t \geq 1$. In the rest of the proof, we assume $0 < \delta < 1$. Note that $\bm x_{t+1} = \frac{t\bm x_t}{t+2}  + \frac{2 \tilde{\bm v}_t}{t+2}$ and $\bm x_1 = \tilde{\bm v}_0$ by Line 7 of Algorithm \ref{algo:approx-fw}. We rewrite $\bm x_{t+1}$ as a convex combination of ${\tilde{\bm v}}_t$ as follows
\begin{equation}
\bm x_{t+1} = \frac{t\bm x_t}{t+2}  + \frac{2 \tilde{\bm v}_t}{t+2} = \sum_{i=0}^{t} \frac{2 (i+1) \tilde{\bm v}_i}{(t+1)(t+2)} =  \sum_{i=0}^{t} \left\{ \frac{ 2 (i+1)}{(t+1)(t+2)}\cdot \frac{-C\nabla f(\bm x_i)_{S_i}}{\|\nabla f(\bm x_i)_{S_i}\|_2} \right\}, \nonumber
\end{equation}
where the last equality follows Line 6 of Algorithm \ref{algo:approx-fw}. Since we assume  $\bm x_0 \in \mathcal{D}$ and when $t\geq 1$, the inner product $\langle - \bm x_t, \nabla f(\bm x_t) \rangle$ can be expressed as 
\begin{align}
\langle - \bm x_t, \nabla f(\bm x_t) \rangle &= \sum_{i=0}^{t-1} \frac{2 (i+1)}{t(t+1)} \left\langle \frac{ C \nabla f(\bm x_i)_{S_i}}{\|\nabla f(\bm x_i)_{S_i}\|_2}, \nabla f(\bm x_t) \right\rangle \label{inequ:thm:5.1-1}
\end{align}

First of all, note that when $\delta \in (0,1/2]$, one can simply apply convexity, that is,
\begin{align*}
h(\bm x_t) &= f(\bm x_t) - f(\bm x_*) \nonumber\\
&\leq |\nabla f(\bm x_t)^\top (\bm x_t - \bm x_*)| \nonumber\\
&\leq |\nabla f(\bm x_t)^\top \bm x_t|  + | \nabla f(\bm x_t)^\top  \bm x_*| \nonumber\\
&\leq C \sqrt{s}\|\nabla f(\bm x_t)\|_\infty + C \sqrt{s}\|\nabla f(\bm x_t)\|_\infty \nonumber\\
&= 2 C\sqrt{s}\|\nabla f(\bm x_t)\|_\infty, \label{inequ:thm:5.1-2}
\end{align*}
where the last inequality is due to (\ref{inequ:thm:5.1-1}) and follows the fact that $\|\bm x_*\|_0 \leq s$.

By $L$-smooth, we have
\begin{eqnarray*}
f(\bm x_{t+1})-f(\bm x_t) &\leq& \eta_t  \underbrace{\nabla f(\bm x_t)^\top \tilde{\bm v}_t}_{\leq \delta \nabla f(\bm x_t)^\top {\bm v}_t} -  \eta_t  \nabla f(\bm x_t)^\top\bm x_t  + \frac{\eta_t^2 L \|\tilde {\bm v}_t-\bm x_t\|_2^2}{2} \\
&\leq& -\eta_t \delta (f(\bm x_t)-f^*) - \eta_t (1-\delta)\nabla f(\bm x_t)^\top \bm x_t + \frac{\eta_t^2 L \|\tilde {\bm v}_t-\bm x_t\|_2^2}{2},
\end{eqnarray*}
where the first inequality is due to $\nabla f(\bm x_t)^\top \tilde{\bm v}_t \leq \delta \nabla f(\bm x_t)^\top {\bm v}_t$ and second inequality follows from convexity, i.e., $\nabla f(\bm x_t)^\top (\bm v_t - \bm x_t) \leq - (f(\bm x_t) - f^* )$. By setting $\eta_t = 2/(t+2)$, we reach
\begin{align}
h({\bm x_{t+1}}) &\leq \left(1-\frac{2\delta}{t+2}\right)h({\bm x_t}) + \frac{2(1-\delta)}{t+2} \left\langle - \bm x_t,\nabla f(\bm x_t) \right\rangle + \frac{2 L \|\tilde{\bm v}_t - \bm x_t\|^2}{(t+2)^2} \nonumber\\
&\leq \left(1-\frac{2\delta}{t+2}\right)h({\bm x_t}) + \frac{2(1-\delta)}{t+2}\sum_{i=0}^{t-1} \frac{2 (i+1)}{t(t+1)} \left\langle \frac{ \nabla f(\bm x_i)_{S_i}}{\|\nabla f(\bm x_i)_{S_i}\|_2}, \nabla f(\bm x_t) \right\rangle + \frac{8 L C^2}{(t+2)^2} \nonumber\\
&= \left(1-\frac{2\delta}{t+2}\right)h({\bm x_t}) + \frac{1}{(t+2)^2}\left\{ \frac{4(1-\delta)(t+2)}{t(t+1)} \left\langle \sum_{i=0}^{t-1} \frac{ (i+1) C \nabla f(\bm x_i)_{S_i}}{\|\nabla f(\bm x_i)_{S_i}\|_2}, \nabla f(\bm x_t) \right\rangle + 8 L C^2 \right\} \\
&= \left(1-\frac{2\delta}{t+2}\right)h({\bm x_t}) + \frac{1}{(t+2)^2}\left\{ 4(1-\delta)(t+2) \left\langle -\bm x_t, \nabla f(\bm x_t) \right\rangle + 8 L C^2 \right\}. \label{inequ:step2}
\end{align}
By applying Lemma \ref{lemma:currence-inequality} with $A_t \triangleq \max_{i \in [t]} 4(1-\delta)(i+2) |\left\langle -\bm x_i, \nabla f(\bm x_i) \right\rangle| + 8 L C^2$ and combine it with (\ref{inequ:thm:5.1-2}), we complete our proof.
\end{proof}

\begin{repcorollary}{corollary:practical-bound}[Practical rate of \textsc{DMO-FW-I}] Suppose $\nu \in (0, 1]$ and $\|\nabla f(\bm x)\|_\infty \leq B$, the primal error of \textsc{DMO-FW-I} admits the following practical bound
\begin{equation}
h(\bm x_t) \leq \begin{cases}
\mathcal{O}\left( \frac{B C\sqrt{s}}{t^{\nu}}\right) & \|\nabla f(\bm x_t)\|_\infty \leq \frac{B}{t^\nu} \\
\mathcal{O}\left(\frac{B C\sqrt{s}(1-\delta)}{\delta}\right) & \|\nabla f(\bm x)\|_\infty \leq B.
\end{cases}
\end{equation}
where $s \triangleq \max_{S\in \mathbb{M}} |S|$ is the maximal allowed sparsity.
\end{repcorollary}

\begin{proof}
First of all, when $\|\nabla f(\bm x_t)\|_\infty \leq B / t^\nu$, one can simply apply Thm. \ref{thm:1:fw-dmo-i} to show $h(\bm x_t)\leq \mathcal{O}(BC\sqrt{s}/t^\nu)$. In the rest, we show the second case when $\|\nabla f(\bm x)\|_\infty \leq B$. Rewrite $\bm x_{t}$ where $\bm x_1 = \tilde{\bm v}_0$ and 
\begin{align*}
\bm x_{t+1} &= \frac{t\bm x_t}{t+2}  + \frac{2 \tilde{\bm v}_t}{t+2} \\
&= \sum_{i=0}^{t} \frac{2 (i+1) \tilde{\bm v}_i}{(t+1)(t+2)} \\
&= \sum_{i=0}^{t} \left\{ \frac{2 (i+1)}{(t+1)(t+2)}\cdot \frac{ - C \nabla f(\bm x_i)_{S_i}}{\|\nabla f(\bm x_i)_{S_i}\|_2} \right\}, \nonumber
\end{align*}
where the last equality follows Line 4 of Algorithm \ref{algo:approx-fw}. When $t=0$, we assume the initial point is $\bm x_0 = \bm 0$. When $t\geq 1$ the inner product $\langle - \bm x_t, \nabla f(\bm x_t) \rangle$ can be bounded as the following
\begin{align*}
\langle - \bm x_t, \nabla f(\bm x_t) \rangle &= C \sum_{i=0}^{t-1} \frac{2 (i+1)}{t(t+1)} \left\langle \frac{ \nabla f(\bm x_i)_{S_i}}{\|\nabla f(\bm x_i)_{S_i}\|_2}, \nabla f(\bm x_t) \right\rangle \\
&\leq C \underbrace{\sum_{i=0}^{t-1} \frac{2 (i+1)}{t(t+1)}}_{=1} \underbrace{\left\| \frac{ \nabla f(\bm x_i)_{S_i}}{\|\nabla f(\bm x_i)_{S_i}\|_2}\right\|_2}_{=1}\cdot \|\nabla f(\bm x_t)_{S_i}\|_2 \\
&\leq C \sqrt{s} \|\nabla f(\bm x_t)\|_\infty,
\end{align*}
where the first inequality follows by the Holder's inequality and the last inequality is the assumption of boundness of $\|\nabla f(\bm x_t)\|_\infty$. By $L$-smooth of $f$, we have
\begin{eqnarray*}
f(\bm x_{t+1})-f(\bm x_t) &\leq& \eta_t  \underbrace{\nabla f(\bm x_t)^\top \tilde {\bm v}_t}_{\leq \delta \nabla f(\bm x_t)^\top {\bm v}_t} -  \eta_t  \nabla f(\bm x_t)^\top\bm x_t  + \frac{L\eta_t^2}{2}\|\tilde {\bm v}_t-\bm x_t\|_2^2 \\
&\leq& -\eta_t \delta (f(\bm x_t)-f^*) - (1-\delta)\eta_t \nabla f(\bm x_t)^\top \bm x_t + \frac{ 4 \eta_t^2 L  C^2}{2}.
\end{eqnarray*}
By setting the step size $\eta_t = \tfrac{2}{t+2}$. This leads to the following
\begin{align}
h({\bm x_{t+1}}) &\leq \left(1-\frac{2\delta}{t+2}\right)h({\bm x_t}) + \frac{2(1-\delta)}{(t+2)} \left\langle - \bm x_t,\nabla f(\bm x_t) \right\rangle + \frac{8 L C^2}{(t+2)^2} \nonumber\\
&\leq \left(1-\frac{2\delta}{t+2}\right)h({\bm x_t}) + \frac{2(1-\delta) \sqrt{s} \|\nabla f(\bm x_t)\|_\infty C}{(t+2)} + \frac{8 L C^2}{(t+2)^2}. \label{inequ:step22}
\end{align}
Notice that the above recurrence (\ref{inequ:step22}) can be written as 
\begin{align*}
h(\bm x_{t+1}) &\leq \prod_{i=0}^t \left( 1- \frac{2\delta}{i+2} \right) h(\bm x_0) + \sum_{i=0}^t \left\{ \left(  \frac{2(1-\delta) \sqrt{s}M_i C}{(i+2)} + \frac{8 L C^2}{(i+2)^2} \right) \prod_{j=i+1}^t \left( 1 - \frac{2\delta}{j+2} \right) \right\} \\
&\leq (1-\delta)\left(\frac{3}{t+3}\right)^{2\delta} \cdot h(\bm x_0) + \sum_{i=0}^t \left\{ \left(  \frac{2(1-\delta) \sqrt{s}M_i C}{(i+2)} + \frac{8 L C^2}{(i+2)^2} \right) \prod_{j=i+1}^t \left( 1 - \frac{2\delta}{j+2} \right) \right\} \\
&\leq (1-\delta)\left(\frac{3}{t+3}\right)^{2\delta} \cdot h(\bm x_0) + \sum_{i=0}^t \left\{ \left(  \frac{2(1-\delta) \sqrt{s}M_i C}{(i+2)} + \frac{8 L C^2}{(i+2)^2} \right)\frac{(i+3)^{2\delta}}{(t+3)^{2\delta}} \right\}.
\end{align*}
Let $M_{\bar{t}} = \max_{i\in [t]} \|\nabla f(\bm x_t)\|_\infty$. Therefore, we reach
\begin{align*}
h(\bm x_{t+1}) &\leq (1-\delta)\left(\frac{3}{t+3}\right)^{2\delta} \cdot h(\bm x_0) + \sum_{i=0}^t \left\{ \left(  \frac{2(1-\delta)\sqrt{s}M_{\bar{t}} C }{(i+2)} + \frac{8 L C^2}{(i+2)^2} \right)\frac{(i+3)^{2\delta}}{(t+3)^{2\delta}} \right\}
\end{align*}
We consider three cases: 

1) When $\delta = 1/2$, we have
\begin{align*}
h(\bm x_{t+1}) &\leq \frac{3}{2(t+3)} \cdot h(\bm x_0) +  \sum_{i=0}^t \frac{\sqrt{s}M_{\bar{t}} C(i+3)}{(i+2)(t+3)} + \sum_{i=0}^t \frac{8 L C^2(i+3)}{(i+2)^2(t+3)} \\
&\leq \frac{3}{2(t+3)} \cdot h(\bm x_0) + \frac{3\sqrt{s}M_{\bar{t}} C}{2(t+3)} (t+1) + \frac{12 L C^2 \ln (t+2)}{(t+3)}\\
&\leq \mathcal{O}(\sqrt{s}M_{\bar{t}} C) \\
&= \mathcal{O}\left(\frac{\sqrt{s}(1-\delta)M_{\bar{t}} C}{\delta}\right).
\end{align*}

2) when $\delta > 1/2$, we have
\begin{align*}
h(\bm x_{t+1}) &\leq (1-\delta)\left(\frac{3}{t+3}\right)^{2\delta} \cdot h(\bm x_0) + \sum_{i=0}^t \left\{ \left(  \frac{2(1-\delta)\sqrt{s}M_{\bar{t}} C }{(i+2)} + \frac{8 L C^2}{(i+2)^2} \right)\frac{(\frac{3}{2})^{2\delta}(i+2)^{2\delta}}{(t+3)^{2\delta}} \right\} \\
&\leq (1-\delta)\left(\frac{3}{t+3}\right)^{2\delta} \cdot h(\bm x_0) + \left(\frac{9}{4}\right)^\delta\int_{2}^{t+3} \frac{\sqrt{s}(1-\delta)M_{\bar{t}}C}{\delta} \frac{x^{2\delta}}{(t+3)^{2\delta}} + \frac{8 L C^2}{(2\delta - 1)} \frac{x^{2\delta-1}}{(t+3)^{2\delta}} d x \\
&\leq (1-\delta)\left(\frac{3}{t+3}\right)^{2\delta} \cdot h(\bm x_0) + \left(\frac{9}{4}\right)^\delta \left(\frac{\sqrt{s}(1-\delta)M_{\bar{t}} C}{\delta} + \frac{8 L C^2}{(2\delta - 1)(t+3)}\right) \\
&\leq \mathcal{O}\left(\frac{\sqrt{s} (1-\delta) M_{\bar{t}}C}{\delta}\right).
\end{align*}

3) when $\delta < 1/2$, we have
\begin{align*}
&\leq (1-\delta)\left(\frac{3}{t+3}\right)^{2\delta} \cdot h(\bm x_0) + \sum_{i=0}^t \left\{ \left(\frac{3\sqrt{s}(1-\delta)M_{\bar{t}} C }{(i+3)} + \frac{18 L C^2}{(i+3)^2} \right)\frac{(i+3)^{2\delta}}{(t+3)^{2\delta}} \right\} \\
&\leq (1-\delta)\left(\frac{3}{t+3}\right)^{2\delta} \cdot h(\bm x_0) + \int_{2}^{t+3} \frac{3\sqrt{s}(1-\delta)M_{\bar{t}} C x^{2\delta}}{2\delta (t+3)^{2\delta}} d x + \int_{2}^{t+3} \frac{18 L C^2 x^{2\delta-1}}{(2\delta-1)(t+3)^{2\delta}} d x \\
&\leq (1-\delta)\left(\frac{3}{t+3}\right)^{2\delta} \cdot h(\bm x_0) + \frac{3(1-\delta)\sqrt{s}M_{\bar{t}} C }{2\delta} + \frac{9 L C^2 4^\delta}{(1 - 2\delta)(t+3)^{2\delta}} \\
&\leq \mathcal{O}\left(\frac{(1-\delta)\sqrt{s}M_{\bar{t}} C}{\delta}\right).
\end{align*}
Combine the above three cases and use the fact that $M_{\bar{t}} \leq B$ by our assumption, we prove the corollary.
\end{proof}

\begin{remark}
The above bound is tight when $\delta = 1$, and it recovers the standard convergence of FW \citep{jaggi2013revisiting}.
\end{remark}

\subsection{Proof of Theorem \ref{thm:2:fw-dmo-ii}}

\begin{reptheorem}{thm:2:fw-dmo-ii}[Convergence of \textsc{DMO-FW-II}]
Under the same assumptions  as in Thm. \ref{thm:1:fw-dmo-i}, for any $t \geq 1$, the primal error of \textsc{DMO-FW-II} satisfies
\begin{equation}
h(\bm x_t) = f(\bm x_{t})-f(\bm x^*) \leq \frac{8 L C^2}{\delta^2 (t+2)}, \label{inequ:rate:thm22}
\end{equation}
where $\bm x_{t} \in \mathcal{D}/\delta$ and $\bm x^* \in \arg\min_{\bm x \in \mathcal{D}} f(\bm x)$.
\end{reptheorem}

\begin{proof}
By $L$-smoothness of $f$, we have 
\[
f(\bm x_{t+1}) \leq f(\bm x_t) + \eta_t \left\langle \nabla f(\bm x_t), \frac{\tilde{\bm v}_t}{\delta} - \bm x_t \right\rangle + \frac{L \eta_t^2 \|\tfrac{\tilde{\bm v}_t}{\delta} - \bm x_t\|^2}{2}.
\]
Let $\eta_t = \tfrac{2}{(t+2)}$ and adding $-f(\bm x^*)$ (notice that $\bm x^* \in \arg\min_{\bm x \in \mathcal{D}} f(\bm x)$) on both sides, we have
\begin{align*}
h(\bm x_{t+1}) \leq h(\bm x_t) + \frac{2}{t+2} \left\langle \nabla f(\bm x_t), \frac{\tilde{\bm v}_t}{\delta} - \bm x_t \right\rangle + \frac{8 L C^2}{\delta^2(t+2)^2},
\end{align*}
where the last term follows from the scaling diameter of $\mathcal{D} / \delta$, i.e. 
\[
\bm x_t,\tilde {\bm v}_t/\delta\in \mathcal D/\delta \Rightarrow \|\bm x_t - \tilde {\bm v}_t/\delta\|^2_2 \leq (2 C / \delta)^2.
\]
Since $\tilde {\bm v}_t$ is a $(\delta, \nabla f(\bm x_t), \mathcal{D})$-DMO, it admits 
\[
\langle \nabla f(\bm x_t), \tilde{\bm v}_t \rangle \leq \delta \min_{\bm s \in \mathcal{D}}   \langle \nabla f(\bm x_t), \bm s \rangle <0.
\]
Scaling by $1 / \delta$ and then adding $\langle \nabla f(\bm x_t), -\bm x_t \rangle$ both sides, we reach 
\[
\langle \nabla f(\bm x_t), \tilde{\bm v}_t /\delta - \bm x_t\rangle \leq \min_{\bm s \in \mathcal{D}} \langle \nabla f(\bm x_t), \bm s - \bm x_t \rangle.
\]
We continue to have the following
\begin{align*}
h(\bm x_{t+1}) &\leq h(\bm x_t) + \frac{2}{t+2} \min_{\bm s \in \mathcal{D}}  \langle \nabla f(\bm x_t), {\bm s} - \bm x_t \rangle + \frac{2 L D^2}{\delta^2(t+2)^2} \\
&\leq h(\bm x_t) + \frac{2}{t+2} \langle \nabla f(\bm x_t), {\bm x}^* - \bm x_t \rangle + \frac{8 L C^2}{\delta^2(t+2)^2} \\
&\leq \left(1-\frac{2}{t+2}\right) h(\bm x_t) + \frac{8 L C^2}{\delta^2(t+2)^2},
\end{align*}
where the last inequality is due to the convexity of $f$, i.e. $\langle \nabla f(\bm x_t), {\bm x}^* - \bm x_t \rangle \leq f(\bm x^*) - f(\bm x_t)$ and $\bm x^* \in \mathcal{D}$. By a similar argument of induction shown in Lemma \ref{lemma:currence-inequality}, we can show the bound of 
\[
h(\bm x_{t}) \leq \frac{8 L C^2}{\delta^2 (t+2)}.
\]
\end{proof}
\begin{remark}
The above proof follows a similar proof strategy as in \citet{jaggi2013revisiting}. Different from previous one, we show that when $\bm x_t$ extended to $\mathcal{D}/\delta$ with a $\delta$-approximation DMO, we can still have a convergence rate inverse proportional to $\delta$.
\end{remark}

\subsection{Proof of Theorem \ref{thm:4:acc-fw-dmo-ii}}

\paragraph{Notations.} Recall that our domain $\mathcal{D} = \operatorname{conv} \{ \bm x: \| \bm x\|_2 \leq C, \operatorname{supp}(\bm x) \in \mathbb{M} \}$. The corresponding set of extreme points is $\mathcal{V}:= \{\bm x: \operatorname{supp}(\bm x) \in \mathbb{M}, \| \bm x\|_2 = C\}$. Follow notations of \citet{garber21a}, we denote the set of optimal points $\mathcal{X}^* := \arg\min_{\bm x\in \mathcal{D}} f(\bm x)$. Recall the quadratic growth condition as the following.
\begin{definition}[Quadratic Growth Condition]
Let $f$ be continuous differentiable.  $\mathcal{X}^* \triangleq \arg\min_{\bm x \in \mathcal{D}} f(\bm x)$. We say $f$ satisfies quadratic growth condition on $\mathcal{D}$ if there exists a constant $\mu > 0$ such that
\begin{equation}
f(\bm x) - f(\bm x^*) \geq \frac{\mu}{2} \| \bm x - \left[\bm x\right]_{\mathcal{X}^*}\|_2^2,
\end{equation}
for all $\bm x \in \mathcal{D}$ where $\left[\bm x\right]_{\mathcal{X}^*} \triangleq \argmin{\bm z \in \mathcal{X}^*} \| \bm z - \bm x\|_2^2$.
\end{definition}
The quadratic growth condition is weaker than restricted strong convex and strongly convex. When $f$ is convex differentiable, it has been observed, it is equivalent to others such as PL condition \cite{karimi2016linear}. We start from the following key lemma.
\begin{lemma}
If $\delta=1$, each iteration of \textsc{DMO-AccFW} admits
\begin{equation}
\eta_t \langle \bm v_t - \bm x_t, \nabla f(\bm x_t)\rangle + \frac{L \eta_t^2}{2} \left\| \bm v_t - \bm x_t\right\|_2^2 \leq \min_{\bm v \in \mathcal{D}} \left\{\eta_t \langle \bm v - \bm x_t, \nabla f(\bm x_t)\rangle + \frac{L \eta_t^2}{2} \left\| \bm v - \bm x_t\right\|_2^2  + \frac{L \eta_t^2}{2} \left(C^2 - \| \bm v\|_2^2\right)\right\}. \label{inequ:lemma:12:general-bound}
\end{equation}
Let $\mathcal{X}^* = \argmin{\bm x \in \mathcal{D}} f(\bm x)$ and suppose optimal points are on the boundary, i.e., $\| \bm x^* \|_2 = C$ for all $\bm x^* \in \mathcal{X}^*$. 
\begin{equation}
\eta_t \langle \bm v_t - \bm x_t, \nabla f(\bm x_t)\rangle + \frac{L \eta_t^2}{2} \left\| \bm v_t - \bm x_t\right\|_2^2 \leq \eta_t \langle [\bm x_t]_{\mathcal{X}^*} - \bm x_t, \nabla f(\bm x_t)\rangle + \frac{L \eta_t^2}{2} \left\| [\bm x_t]_{\mathcal{X}^*} - \bm x_t\right\|_2^2,
\end{equation}
where $\bm v_t$ is a minimizer (recall that $\bar{\bm v}_t = \bm v_t$ when $\delta=1$), e.g.,
\[
\left\langle \bm v_t, -\left( \bm x_t - \frac{\nabla f(\bm x_t)}{L \eta_t} \right) \right\rangle \leq \min_{\bm v \in \mathcal{D} }\left\langle \bm v, -\left( \bm x_t - \frac{\nabla f(\bm x_t)}{L \eta_t} \right) \right\rangle
\]
and where here $\left[\bm x\right]_{\mathcal{X}^*} \triangleq \argmin{\bm z \in \mathcal{X}^*} \| \bm z - \bm x\|_2^2$. 
\label{lemmma:equivilent-lemma}
\end{lemma}
\begin{proof}
Since $\bm v_t$ follows the above inequality, we have
\[
L \eta_t^2 \left\langle \bm v_t, -\left( \bm x_t - \frac{\nabla f(\bm x_t)}{L \eta_t} \right) \right\rangle \leq \min_{\bm v \in \mathcal{D} } L \eta_t^2  \left\langle \bm v, -\left( \bm x_t - \frac{\nabla f(\bm x_t)}{L \eta_t} \right) \right\rangle
\]
By adding both sides $L \eta_t^2 (C^2 + \|\bm x_t - \frac{\nabla f(\bm x_t)}{L \eta_t}\|_2^2)/2$, we reach
\[
\frac{L \eta_t^2}{2} \left\| \bm v_t -\left( \bm x_t - \frac{\nabla f(\bm x_t)}{L \eta_t} \right) \right\|_2^2 \leq \min_{\bm v \in \mathcal{D} } \left\{ \frac{L \eta_t^2}{2}  \left\| \bm v -\left( \bm x_t - \frac{\nabla f(\bm x_t)}{L \eta_t} \right) \right\|_2^2 + \frac{L \eta_t^2}{2} \left(C^2 - \| \bm v\|_2^2\right) \right\},
\]
where $\|\bm v_t\|_2^2 = C^2$. Since $\left[\bm x_t\right]_{\mathcal{X}^*} \in \mathcal{X}^* \subseteq \mathcal{D}$ and $\|\left[\bm x_t\right]_{\mathcal{X}^*}\|_2^2 = C^2$ by our assumption, we immediately get
\begin{align*}
\frac{L \eta_t^2}{2} & \left\| \bm v_t - \bm x_t\right\|_2^2 + \eta_t \langle \bm v_t - \bm x_t, \nabla f(\bm x_t)\rangle + \frac{\left\|\nabla f(\bm x_t)\right\|_2^2}{2 L} \\
&\leq \min_{\bm v \in \mathcal{D}} \left\{\frac{L \eta_t^2}{2} \left\| \bm v - \bm x_t\right\|_2^2 + \eta_t \langle \bm v - \bm x_t, \nabla f(\bm x_t)\rangle + \frac{\left\|\nabla f(\bm x_t)\right\|_2^2}{2 L} + \frac{L \eta_t^2}{2} \left(C^2 - \| \bm v\|_2^2\right)\right\} \\
&\leq \frac{L \eta_t^2}{2} \left\| [\bm x_t]_{\mathcal{X}^*} - \bm x_t\right\|_2^2 + \eta_t \langle [\bm x_t]_{\mathcal{X}^*} - \bm x_t, \nabla f(\bm x_t)\rangle + \frac{\left\|\nabla f(\bm x_t)\right\|_2^2}{2 L}.
\end{align*}
We immediately get
\begin{align*}
\frac{L \eta_t^2}{2} & \left\| \bm v_t - \bm x_t\right\|_2^2 + \eta_t \langle \bm v_t - \bm x_t, \nabla f(\bm x_t)\rangle + \frac{\left\|\nabla f(\bm x_t)\right\|_2^2}{2 L} \\
&\leq \min_{\bm v \in \mathcal{D}} \left\{\frac{L \eta_t^2}{2} \left\| \bm v - \bm x_t\right\|_2^2 + \eta_t \langle \bm v - \bm x_t, \nabla f(\bm x_t)\rangle + \frac{\left\|\nabla f(\bm x_t)\right\|_2^2}{2 L} + \frac{L \eta_t^2}{2} \left(C^2 - \| \bm v\|_2^2\right)\right\}.
\end{align*}
Simplifying the above inequality, we have the general bound (\ref{inequ:lemma:12:general-bound}). When $\|\bm x^*\|_2^2 = C^2$, simplifying the above inequality, we have
\[
\eta_t \langle \bm v_t - \bm x_t, \nabla f(\bm x_t)\rangle + \frac{L \eta_t^2}{2} \left\| \bm v_t - \bm x_t\right\|_2^2 \leq \eta_t \langle [\bm x_t]_{\mathcal{X}^*} - \bm x_t, \nabla f(\bm x_t)\rangle + \frac{L \eta_t^2}{2} \left\| [\bm x_t]_{\mathcal{X}^*} - \bm x_t\right\|_2^2.
\]

\end{proof}

\begin{theorem}
Let $f$ be convex and satisfies quadratic growth condition. Assume that for all $\bm x^*$, $\| \bm x^*\|_2^2 = C$, then if $\delta=1$ and $\eta_t = 2/(t+2)$, \textsc{DMO-AccFW-I} has the following convergence rate 
\begin{equation}
h(\bm x_{t}) \leq \frac{ 4 e^{4 L/\mu} }{(t+2)^2} h(\bm x_0)
\end{equation}
for all $t \geq 1$.
\end{theorem}
\begin{proof}
By the $L$-smoothness of $f$, we have
\begin{align*}
f(\bm x_{t+1}) &\leq f(\bm x_t) + \nabla f(\bm x_t)^\top (\bm x_{t+1} - \bm x_t) + \frac{L}{2} \| \bm x_{t+1} - \bm x_t \|_2^2 \\
&= f(\bm x_t) + \eta_t \nabla f(\bm x_t)^\top \left(\tilde{\bm v}_t - \bm x_t\right) + \frac{L \eta_t^2}{2} \left\| \tilde{\bm v}_t - \bm x_t \right\|_2^2\\
&\overset{(a)}{=} f(\bm x_t) + \eta_t \nabla f(\bm x_t)^\top \left({\bm v}_t - \bm x_t\right) + \frac{L \eta_t^2}{2} \left\| {\bm v}_t - \bm x_t \right\|_2^2\\
&\overset{(b)}{\leq} f(\bm x_t) + \eta_t \nabla f(\bm x_t)^\top \left(\left[{\bm x}_t\right]_{\mathcal{X}^*} - \bm x_t\right) + \frac{L \eta_t^2}{2} \left\| \left[ {\bm x}_t \right]_{\mathcal{X}^*} - \bm x_t \right\|_2^2 \\
&\overset{(c)}{\leq} f(\bm x_t) - \eta_t \left(f(\bm x_t) - f(\left[ {\bm x}_t \right]_{\mathcal{X}^*})\right) + \frac{L \eta_t^2}{2} \left\| \left[ {\bm x}_t \right]_{\mathcal{X}^*} - \bm x_t \right\|_2^2,
\end{align*}
where (a) is due to $\delta=1$, (b) follows from Lemma \ref{lemmma:equivilent-lemma}, and (c) uses the convexity of $f$. By setting the step size $\eta_t = 2/(t+2)$ and add $-f(\bm x^*)$ on both sides, we reach at
\[
h(\bm x_{t+1}) \leq \left(1-\frac{2}{t+2}\right) h(\bm x_t) + \frac{2 L \| \bm x_t - \left[\bm x_t\right]_{\mathcal{X}^*} \|_2^2}{ (t+2)^2}.
\]
By the assumption of quadratic growth property property of $f$, we have 
\[
f(\bm x_t) - f(\bm x^*) \geq \frac{\mu}{2} \| \bm x_t - \left[\bm x \right]_{\mathcal{X}^*}\|_2^2.
\]
Specifically, for all $t \geq 0$
\[
\| \bm x_t - \left[\bm x \right]_{\mathcal{X}^*}\|_2^2 \leq \frac{2}{\mu}\left( f(\bm x_t) - f(\left[\bm x \right]_{\mathcal{X}^*}) \right) = \frac{2}{\mu} h(\bm x_t).
\]
We have
\begin{align*}
h(\bm x_{t+1}) &\leq \left(1-\frac{2}{t+2}\right) h(\bm x_t) + \frac{2 L \| \bm x_t - \left[\bm x \right]_{\mathcal{X}^*}\|_2^2}{ (t+2)^2} \\
&\leq \left(1-\frac{2}{t+2}\right) h(\bm x_t) + \frac{4 L h(\bm x_t)}{\mu (t+2)^2} \\
&= \left(1-\frac{2}{t+2} + \frac{4 L}{\mu (t+2)^2} \right) h(\bm x_t) \\
&= \prod_{i=0}^t \left(1-\frac{2}{i+2} + \frac{4 L}{\mu (i+2)^2} \right) h(\bm x_0) \\
&\leq \exp\left\{\sum_{i=0}^t \left(- \frac{2}{i+2} + \frac{4L}{\mu (i+2)^2}\right)\right\} \cdot h(\bm x_0) \\ 
&\leq \frac{4 e^{4 L/\mu}}{(t+3)^2} \cdot h(\bm x_0).
\end{align*}
\end{proof}

\begin{theorem}
Let $f$ be convex and satisfies quadratic growth condition. Assume that 
$D_* = \min_{\bm x^* \in \mathcal{X}^*} \| \bm x^*\|_2$, then if $\delta=1$ and $\eta_t = 2/(t+2)$, \textsc{DMO-AccFW-I} has the following convergence rate 
\begin{equation}
h(\bm x_{t}) \leq \min\left\{ \frac{3 L e^{2 L / \mu} (C^2 -D_*^2)}{t+2} + \frac{4 e^{4 L/\mu} h(\bm x_0)}{(t+2)^2}, \frac{ 2 L(5 C^2 - D_*^2)}{(t+2)}\right\}
\end{equation}
for all $t \geq 1$.
\end{theorem}
\begin{proof}
By setting $\delta=1$, by $L$-smooth, we have
\begin{align*}
f(\bm x_{t+1}) &\leq f(\bm x_t) + \eta_t \langle \nabla f(\bm x_t), \bm v_t - \bm x_t \rangle + \frac{L \eta_t^2}{2} \| \bm v_t - \bm x_t \|_2^2 \\
&\leq f(\bm x_t) +  \min_{\bm v \in \mathcal{D}} \left\{\eta_t \langle \nabla f(\bm x_t), \bm v - \bm x_t \rangle + \frac{L \eta_t^2}{2} \| \bm v - \bm x_t \|_2^2 + \frac{L \eta_t^2}{2}(C^2 - \|\bm v\|_2^2)\right\} \\
&\leq f(\bm x_t) +  \eta_t \langle \nabla f(\bm x_t), \bm x^* - \bm x_t \rangle + \frac{L \eta_t^2}{2}\|\bm x^* - \bm x_t \|_2^2 + \frac{L \eta_t^2}{2}\left(C^2 - \|\bm x^*\|_2^2\right).
\end{align*}
That is,
\begin{equation}
h(\bm x_{t+1}) \leq (1 - \eta_t) h(\bm x_t) + \frac{L \eta_t^2}{2} \left(\|\bm x^* - \bm x_t \|_2^2 + C^2 - \|\bm x^*\|_2^2\right). \label{inequ:thm:15-1}
\end{equation}
By the quadratic growth condition, we have
\[
h(\bm x_{t+1}) \leq (1 - \eta_t) h(\bm x_t) + \frac{L \eta_t^2 h(\bm x_t)}{\mu}  + \frac{L \eta_t^2}{2} \left(\|\bm v_t\|_2^2 - \|\bm x^*\|_2^2\right).
\]
Let $\eta_t = \tfrac{2}{(t+2)}$ and $D_* =\|\bm x^*\|_2$, we have
\begin{align*}
h(\bm x_{t+1}) &\leq \left( 1 - \frac{2}{t+2} + \frac{4 L}{\mu (t+2)^2}\right) h(\bm x_t) + \frac{2 L\left( C^2 - D_*^2 \right)}{(t+2)^2}  \\
&\leq \prod_{i=0}^t \left(1-\frac{2}{i+2} + \frac{4 L}{\mu (i+2)^2} \right) h(\bm x_0) +\sum_{i=0}^t \left\{ \frac{2 L (C^2 - D_*^2) }{ (i+2)^2}   \cdot \prod_{j=i+1}^t \left(1-\frac{2}{j+2} + \frac{4 L}{\mu (j+2)^2} \right) \right\} \\
&\leq \frac{4 e^{4 L/\mu}}{(t+3)^2} h(\bm x_0) + \sum_{i=0}^t \left\{ \frac{2 L (C^2 - D_*^2)}{ (i+2)^2} \cdot \prod_{j=i+1}^t \left(1-\frac{2}{j+2} + \frac{4 L}{\mu (j+2)^2} \right) \right\} \\
&\leq \frac{4 e^{4 L/\mu}}{(t+3)^2} h(\bm x_0) + \sum_{i=0}^t \left\{ \frac{2 L (C^2-D_*^2)}{ (i+2)^2} \frac{(i+3)^2}{(t+3)^2} \cdot \operatorname{exp}\left\{ \frac{4 L}{\mu} \left(\frac{1}{i+2} - \frac{1}{t+2}\right)\right\}\right\} \\
&\leq \frac{4 e^{4 L/\mu}}{(t+3)^2} h(\bm x_0) + e^{2 L / \mu} \sum_{i=0}^t \frac{2 L (C^2 -D_*^2) (i+3)^2}{ (i+2)^2(t+3)^2} \\
&\leq \frac{4 e^{4 L/\mu}}{(t+3)^2} h(\bm x_0) + 2 L e^{2 L / \mu} (C^2 -D_*^2) \frac{t+1 + 2\ln(t+2) + 1}{(t+3)^2} \\
&\leq \frac{4 e^{4 L/\mu}}{(t+3)^2} h(\bm x_0) + \frac{3 L e^{2 L / \mu} (C^2 -D_*^2)}{t+3}.
\end{align*}
On the other hand, from (\ref{inequ:thm:15-1}), we have
\begin{equation}
h(\bm x_{t+1}) \leq (1 - \frac{2}{t+2}) h(\bm x_t) + \frac{ 2 L(5 C^2 - D_*^2)}{(t+2)^2}.
\end{equation}
Hence, for all $t \geq 1$, we have
\[
h(\bm x_t) \leq \frac{ 2 L(5 C^2 - D_*^2)}{(t+2)}.
\]
Combine these two bounds, we prove the theorem.
\end{proof}

\begin{reptheorem}{thm:5.9}[Practical rate of \textsc{DMO-AccFW-I}]
When $\delta \in (0,1)$ and $\| \nabla f(\bm x)\|_\infty \leq B$ , then \textsc{DMO-AccFW-I} admits
\begin{equation}
h(\bm x_t) \leq \mathcal{O}\left( \sqrt{s}B C(1-\delta)/\delta \right). \nonumber
\end{equation}
Moreover, when $\| \nabla f(\bm x)\|_\infty \leq B/t^{\nu}$ with $\nu \in (0,1]$, we have
\begin{equation}
h(\bm x_t) \leq \mathcal{O}\left(\sqrt{s} B C/ t^\nu \right). \nonumber
\end{equation}
\end{reptheorem}
\begin{proof}
Notice that, at $t$-th iteration, by the property of $(\delta,-\left(\bm x_t - \nabla f(\bm x_t)/(L \eta_t)\right), \mathcal{D})$-DMO operator, we have the following
\begin{align*}
\left\langle \tilde{\bm v}_t, - \left( \bm x_t - \frac{\nabla f(\bm x_t)}{L \eta_t}\right) \right\rangle \leq \delta \cdot \min_{\bm v \in \mathcal{D}} \left\langle \bm v, -\left(\bm x_t - \frac{\nabla f(\bm x_t)}{L \eta_t}\right) \right\rangle,
\end{align*}
where the inequality is due to $(\delta,-\bm x_t + \nabla f(\bm x_t)/(L \eta_t),\mathcal{D})$-DMO oracle in Line 4 of Algorithm \ref{algo:approx-fw}. We continue to have
\begin{align*}
&L \eta_t^2 \left\langle \tilde{\bm v}_t, - \left( \bm x_t - \frac{\nabla f(\bm x_t)}{L \eta_t}\right) \right\rangle \leq \delta \cdot \min_{\bm v \in \mathcal{D}} L \eta_t^2 \left\langle \bm v, -\left(\bm x_t - \frac{\nabla f(\bm x_t)}{L \eta_t}\right) \right\rangle \\
\Leftrightarrow \quad&\frac{C^2 L \eta_t^2}{2} + L \eta_t^2 \left\langle \tilde{\bm v}_t, - \left( \bm x_t - \frac{\nabla f(\bm x_t)}{L \eta_t}\right) \right\rangle + \frac{L \eta_t^2}{2} \left\| \bm x_t - \frac{\nabla f(\bm x_t)}{L \eta_t} \right\|_2^2 \\
&\leq \delta \cdot \min_{\bm v \in \mathcal{D}} \left\{  \frac{L \eta_t^2\|\bm v\|_2^2}{2} +  L \eta_t^2 \left\langle \bm v, -\left(\bm x_t - \frac{\nabla f(\bm x_t)}{L \eta_t}\right) \right\rangle + \frac{L \eta_t^2}{2} \left\| \bm x_t - \frac{\nabla f(\bm x_t)}{L \eta_t} \right\|_2^2 -  \frac{L \eta_t^2 \|\bm v\|_2^2}{2} \right\} \\
&\quad + \frac{C^2 L \eta_t^2}{2}  + \frac{L \eta_t^2(1-\delta)}{2} \left\| \bm x_t - \frac{\nabla f(\bm x_t)}{L \eta_t}\right\|_2^2.
\end{align*}
That is,
\begin{align}
&\frac{L \eta_t^2}{2} \| \tilde{\bm v}_t - \bm x_t \|_2^2 + \eta_t \langle \tilde{\bm v_t} - \bm x_t, \nabla f(\bm x_t) \rangle + \frac{\| \nabla f(\bm x_t) \|_2^2}{2 L} \nonumber\\
&\leq \delta \cdot\min_{\bm v \in \mathcal{D}} \left\{ \frac{L \eta_t^2}{2} \| {\bm v} - \bm x_t \|_2^2 + \eta_t \langle {\bm v} - \bm x_t, \nabla f(\bm x_t) \rangle + \frac{\| \nabla f(\bm x_t) \|_2^2}{2 L} - \frac{L \eta_t^2 \|\bm v\|_2^2}{2} \right\}\nonumber\\
&\quad + \frac{C^2 L \eta_t^2}{2}  + \frac{L \eta_t^2(1-\delta)}{2} \left\| \bm x_t - \frac{\nabla f(\bm x_t)}{L \eta_t}\right\|_2^2 \nonumber\\
\Leftrightarrow \quad& \frac{L \eta_t^2}{2} \| \tilde{\bm v}_t - \bm x_t \|_2^2 + \eta_t \langle \tilde{\bm v_t} - \bm x_t, \nabla f(\bm x_t) \rangle \nonumber\\
&\leq \delta \cdot\min_{\bm v \in \mathcal{D}} \left\{ \frac{L \eta_t^2}{2} \| {\bm v} - \bm x_t \|_2^2 + \eta_t \langle {\bm v} - \bm x_t, \nabla f(\bm x_t) \rangle - \frac{L \eta_t^2 \|\bm v\|_2^2}{2} \right\}\nonumber\\
&\quad + \frac{C^2 L \eta_t^2}{2}  + (1-\delta)\left( \frac{L \eta_t^2}{2} \left\| \bm x_t - \frac{\nabla f(\bm x_t)}{L \eta_t}\right\|_2^2 - \frac{\| \nabla f(\bm x_t)\|_2^2}{2 L}\right) \label{inequ:approximate-inequality}
\end{align}
Define $Q_t = (1-\delta)\left( \frac{L \eta_t^2}{2} \left\| \bm x_t - \frac{\nabla f(\bm x_t)}{L \eta_t}\right\|_2^2 - \frac{\| \nabla f(\bm x_t)\|_2^2}{2 L}\right).$ By the $L$-smooth of $f$, we have
\begin{align*}
f(\bm x_{t+1}) &\leq f(\bm x_t) + \nabla f(\bm x_t)^\top (\bm x_{t+1} - \bm x_t) + \frac{L}{2} \| \bm x_{t+1} - \bm x_t \|_2^2 \\
&= f(\bm x_t) + \eta_t \nabla f(\bm x_t)^\top \left(\tilde{\bm v}_t - \bm x_t\right) + \frac{L \eta_t^2}{2} \left\| \tilde{\bm v}_t - \bm x_t \right\|_2^2\\
&\overset{(a)}{\leq} f(\bm x_t) + \delta\cdot \min_{\bm v \in \mathcal{D}} \left\{ \eta_t \nabla f(\bm x_t)^\top \left(\bm v - \bm x_t\right) + \frac{L \eta_t^2}{2} \left\| \bm v - \bm x_t \right\|_2^2 - \frac{\delta L \eta_t^2 \| \bm v\|_2^2}{2} \right\} + \frac{C^2 L \eta_t^2}{2} + Q_t\\
&\overset{(b)}{\leq} f(\bm x_t) + \delta \cdot \eta_t \nabla f(\bm x_t)^\top \left(\bm x^* - \bm x_t\right) + \delta \cdot  \frac{L \eta_t^2}{2} \left\| \bm x^* - \bm x_t \right\|_2^2 + \frac{C^2 L \eta_t^2}{2} - \frac{\delta L \eta_t^2 \| \bm x^*\|_2^2}{2} + Q_t \\
&\overset{(c)}{\leq} f(\bm x_t) - \delta \cdot \eta_t \left(f(\bm x_t) - f(\bm x^*)\right) + \delta \cdot \frac{L \eta_t^2}{2} \left\| \bm x^* - \bm x_t \right\|_2^2 + \frac{C^2 L \eta_t^2}{2} - \frac{\delta L \eta_t^2 \| \bm x^* \|_2^2}{2} + Q_t \\
\Leftrightarrow h(\bm x_{t+1}) &\leq \left(1- \delta \eta_t + \frac{\delta L \eta_t^2}{\mu} \right) h(\bm x_t) + \frac{L \eta_t^2(C^2 - \delta D_*^2)}{2} + Q_t,
\end{align*}
where (a) is due to (\ref{inequ:approximate-inequality}), (b) follows by $\bm x^*\in \mathcal{D}$, and (c) uses the convexity property. The last inequality is due to the quadratic growth condition and letting step size $\eta_t  = 2/(t+2)$. Hence, we reach the following recurrence relation
\[
h(\bm x_{t+1}) \leq \left(1-\frac{2\delta}{t+2} + \frac{4 \delta L}{\mu (t+2)^2} \right) h(\bm x_t) + \frac{2 L (C^2 - \delta D_*^2)}{(t+2)^2} + (1-\delta)\left( \frac{2 L \| \bm x_t\|_2^2}{(t+2)^2}  + \frac{2 \langle - \bm x_t, \nabla f(\bm x_t) \rangle}{(t+2)}  \right)
\]
Notice that 
\[
\langle - \bm x_t, \nabla f(\bm x_t) \rangle = \sum_{i=0}^{t-1} \frac{2 (i+1)}{t(t+1)} \left\langle \frac{ {{\bm z}_t}_{S_i}}{\|{{\bm z}_t}_{S_i}\|_2}, \nabla f(\bm x_t) \right\rangle,
\]
where $\bm z_t = - C \left(\bm x_t - \frac{\nabla f(\bm x_t)}{L \eta_t}\right)$. We continue to have
\begin{align*}
\langle - \bm x_t, \nabla f(\bm x_t) \rangle &= \sum_{i=0}^{t-1} \frac{2 (i+1)}{t(t+1)} \left\langle \frac{ {{\bm z}_t}_{S_i}}{\|{{\bm z}_t}_{S_i}\|_2}, \nabla f(\bm x_t) \right\rangle \\
&\leq \sum_{i=0}^{t-1} \frac{2 (i+1)}{t(t+1)} \| \nabla f(\bm x_t)_{S_i}\|_2 \leq \sqrt{s} C \|\nabla f(\bm x_t)\|_\infty,
\end{align*}
where $s$ is the maximum number of nonzeros allowed in the sparsity pattern defined by $\mathbb{M}$. We have
\[
h(\bm x_{t+1}) \leq \left(1-\frac{2\delta}{t+2} + \frac{4\delta L}{\mu (t+2)^2} \right) h(\bm x_t) + \frac{2 L (C^2 - \delta D_*^2)}{(t+2)^2} + (1-\delta)\left( \frac{2 L C^2}{ (t+2)^2}  + \frac{2 \sqrt{s} C M_t}{(t+2)}  \right).
\]
Again, let $T = 2 L (C^2 - \delta D_*^2) + 2 L C^2(1-\delta)$ using the recurrence relation, we have
\begin{align*}
h(\bm x_{t+1}) &\leq \prod_{i=0}^t \left(1-\frac{2\delta}{i+2} + \frac{4 \delta L} {\mu (i+2)^2} \right) h(\bm x_0) \\
&\quad + \sum_{i=0}^t \left\{ \left( \frac{T}{(i+2)^2} + \frac{2 \sqrt{s} C M_i (1-\delta)}{(i+2)}  \right)\cdot \prod_{j=i+1}^t \left(1-\frac{2\delta}{j+2} + \frac{4 \delta L}{\mu (j+2)^2} \right) \right\} \\
&\leq \frac{4 e^{4 \delta L/\mu}}{(t+3)^{2\delta}} h(\bm x_0) +  \sum_{i=0}^t \left\{ \left( \frac{T}{ (i+2)^2}  + \frac{2 \sqrt{s} C M_i(1-\delta)}{(i+2)}  \right)\cdot \prod_{j=i+1}^t \left(1-\frac{2\delta}{j+2} + \frac{4 \delta L}{\mu (j+2)^2} \right) \right\} \\
&\leq \frac{4 e^{4 \delta L/\mu}}{(t+3)^{2\delta}} h(\bm x_0) +  \sum_{i=0}^t \left\{ \left( \frac{T}{ (i+2)^2}  + \frac{2 \sqrt{s} C M_i(1-\delta)}{(i+2)}  \right) \left( \frac{(i+3)^{2\delta}}{(t+3)^{2\delta}} \cdot \operatorname{exp}\left\{ \frac{4 \delta L}{\mu} \left(\frac{1}{i+2} - \frac{1}{t+2}\right)\right\} \right) \right\} \\
&\leq \frac{4 e^{4 \delta L/\mu}}{(t+3)^{2\delta}} h(\bm x_0) +  e^{2\delta L / \mu} \sum_{i=0}^t \left\{ \frac{T (i+3)^{2\delta}}{ (i+2)^2(t+3)^{2\delta}}  + \frac{2 \sqrt{s} C M_i (1-\delta) (i+3)^{2\delta}}{(i+2)(t+3)^{2\delta}} \right\} \\
&\leq \frac{4 e^{4 \delta L/\mu}}{(t+3)^{2\delta}} h(\bm x_0) + e^{2\delta L / \mu} \mathcal{O}\left(\frac{2 \sqrt{s} C M_i (1-\delta)}{2\delta}\right) \\
&\leq \mathcal{O}\left( \frac{\sqrt{s} C M_{\bar{t}}(1-\delta)}{\delta}\right) = \leq \mathcal{O}\left( \frac{\sqrt{s} B C(1-\delta)}{\delta}\right), 
\end{align*}
where $M_{\bar{t}} = \max_{i \in [t]} \|\nabla f(\bm x_t)\|_\infty$. The above analysis indicates that the worse case bound is $\mathcal{O}(\sqrt{s}M_{\max})$. However, when $M_t \leq B/t^{\mu}$ with $\mu \in (0,1]$, we have an optimistic bound
\begin{equation}
h(\bm x_t) \leq \mathcal{O}\left(\frac{\sqrt{s} B C}{t^\mu}\right).
\end{equation}
Combine two above bounds, we finish the proof.
\end{proof}

\begin{reptheorem}{thm:5.10}[Convergence of \textsc{DMO-AccFW-II}]
When $\delta \in (0,1)$, then \textsc{DMO-AccFW-II} finds $\bm x_t \in \mathcal{D} /\delta$ and admits 
\begin{equation}
h(\bm x_t) \leq \frac{4 e^{4 L /\mu}}{(t+3)^2} h(\bm x_0) + \frac{28 L^2 (C^2/\delta^2 - D_*^2)}{ 5 \mu (t+3)}.
\end{equation}
\end{reptheorem}
\begin{proof}
By $L$-smooth, we have
\begin{align}
f(\bm x_{t+1}) &\leq f(\bm x_t) + \nabla f(\bm x_t)^\top(\bm x_{t+1} - \bm x_t) + \frac{L}{2}\|\bm x_{t+1} - \bm x_t \|_2^2 \nonumber\\
&\leq f(\bm x_t) + \eta_t \nabla f(\bm x_t)^\top\left(\frac{\tilde{\bm v}_t}{\delta} - \bm x_t\right) + \frac{L \eta_t^2}{2}\left\| \frac{\tilde{\bm v}_t}{\delta} - \bm x_t \right\|_2^2 \label{inequ:thm:20-1}
\end{align}
By the $(\delta, -(\bm x_t - \tfrac{ \nabla f(\bm x_t) }{L \eta_t}))${-DMO}, we have
\[
\left\langle \tilde{\bm v}_t, - \bm x_t +  \frac{ \nabla f(\bm x_t) }{L \eta_t} \right\rangle \leq \delta \cdot \min_{\bm v \in \mathcal{D}} \left\langle {\bm v}, - \bm x_t + \frac{ \nabla f(\bm x_t) }{L \eta_t} \right\rangle
\]
We have the following equivalent form
\[
\frac{L \eta_t^2}{2} \left\| \frac{\tilde{\bm v}_t}{\delta} - \bm x_t + \frac{ \nabla f(\bm x_t) }{L \eta_t} \right\|_2^2 \leq \min_{\bm v \in \mathcal{D}} \left\{\frac{L \eta_t^2}{2} \left\| \bm v - \bm x_t + \frac{ \nabla f(\bm x_t) }{L \eta_t} \right\|_2^2 + \frac{L \eta_t^2 C^2}{2\delta^2} - \frac{L \eta_t^2}{2} \| \bm v\|_2^2 \right\},
\]
which could be simplified as the following
\[
\eta_t \left\langle \nabla f(\bm x_t), \frac{\tilde{\bm v}_t}{\delta} - \bm x_t \right\rangle + \frac{L\eta_t^2}{2} \left\| \frac{\tilde{\bm v}_t}{\delta} - \bm x_t\right\|_2^2 \leq \min_{\bm v \in \mathcal{D}} \left\{ \eta_t \langle \nabla f(\bm x_t), \bm v - \bm x_t \rangle + \frac{L\eta_t^2}{2}\left\|\bm v - \bm x_t\right\|_2^2 + \frac{L \eta_t^2 C^2}{2\delta^2} - \frac{L \eta_t^2}{2} \| \bm v\|_2^2 \right\}.
\]
Therefore, we continue to have
\begin{align*}
f(\bm x_{t+1}) &\leq f(\bm x_t) + \min_{\bm v \in \mathcal{D}} \left\{ \eta_t \langle \nabla f(\bm x_t), \bm v - \bm x_t \rangle + \frac{L\eta_t^2}{2}\left\|\bm v - \bm x_t\right\|_2^2 + \frac{L \eta_t^2 C^2}{2\delta^2} - \frac{L \eta_t^2}{2} \| \bm v\|_2^2 \right\} \\
&\leq f(\bm x_t) - \eta_t (f(\bm x_t) - f(\bm x^*)) + \frac{L\eta_t^2}{2}\left\|\bm x^* - \bm x_t\right\|_2^2 + \frac{L \eta_t^2 C^2}{2\delta^2} - \frac{L \eta_t^2}{2} \| \bm x^*\|_2^2.
\end{align*}
Set $\eta_t = 2/(t+2)$ and use the $\mu$-quadratic growth condition, we have
\[
h(\bm x_{t+1}) \leq \left( 1 - \frac{2}{t+2} + \frac{2 L }{\mu (t+2)^2} \right) h(\bm x_{t}) + \frac{2L}{(t+2)^2} \left(\frac{C^2}{\delta^2} - D_*^2\right).
\]
Letting $B = 2 L (\tfrac{C^2}{\delta^2} - D_*^2)$ and applying the above recurrence relation, we continue to have
\begin{align*}
h(\bm x_{t+1}) &\leq \prod_{i=0}^t \left(1-\frac{2}{i+2} + \frac{4 L} {\mu (i+2)^2} \right) h(\bm x_0) + \sum_{i=0}^t \left\{ \frac{B}{(i+2)^2} \cdot \prod_{j=i+1}^t \left(1-\frac{2}{j+2} + \frac{4 L}{\mu (j+2)^2} \right) \right\} \\
&\leq \frac{4 e^{4 L /\mu}}{(t+3)^2} h(\bm x_0) + \sum_{i=0}^t \frac{2 L B (i+3)^2}{ \mu (i+2)^2 (t+3)^2} \\
&\leq \frac{4 e^{4 L /\mu}}{(t+3)^2} h(\bm x_0) + \frac{2 L B ( t+2 + \ln(t+2))}{ \mu (t+3)^2} \\
&\leq \frac{4 e^{4 L /\mu}}{(t+3)^2} h(\bm x_0) + \frac{14 L B }{ 5 \mu (t+3)},
\end{align*}
where we use the fact that $t+2 + \ln(t+2) \leq 7(t+3) / 5$. Replacing $B$ by $2 L \left(\tfrac{C^2}{\delta^2} - D_*^2\right)$, we prove the theorem.
\end{proof}

\subsection{High probability parameter estimation}

\begin{corollary}[High probability parameter estimation]
Let $h(\bm x_t)$ be primal error for \textsc{DMO-FW} or \textsc{DMO-AccFW}. The estimation error of the graph-structured linear sensing problem admits
\begin{equation}
\| \tilde{\bm x}^* - \bm x_t \|_2 \leq \sqrt{\frac{2 \sqrt{s}C \| \nabla f(\tilde{\bm x}^*)\|_\infty}{\mu}} + \sqrt{\frac{2 h(\bm x_t)}{\mu}}, 
\end{equation}
Moreover, with large enough $n$, there exists an universal constant $c$ such that with high probability:
\[
\| \tilde{\bm x}^* - \bm x_t \|_2 \leq \left( \sqrt{\frac{4 \sigma C \sqrt{s \log d / n}}{\frac{1}{2}} - \frac{c s \log d}{n}} + \sqrt{\frac{2 h(\bm x_t)}{\frac{1}{2} - \frac{c s \log d}{n}}} \right).
\]
\end{corollary}

\begin{proof}
Let $\bm x^*$ be the optimal solution of the graph-structured linear regression problem, that is,
\begin{equation}
\bm x^* \in \argmin{ \bm x \in \mathcal{D}(C,\mathbb{M})} f(\bm x) \triangleq \| \bm A \bm x - \bm y\|_2^2,
\end{equation}
where the measurements are obtained by $\bm y = \bm A \tilde{\bm x}^* + \bm e$ where each $a_{i j} \sim \mathcal{N}(0,1)/\sqrt{n}$ and $ e_i \sim \mathcal{N}(0,\sigma^2)$ for all $i \in [d]$. Notice that, let $\epsilon_t$ be the established bound of $h(\bm x_t)$ for \textsc{DMO-FW} and \textsc{DMO-AccFW}. By $\mu$-quadratic growth condition, we have
\[
f(\tilde{\bm x}^*) \leq f(\bm x_t) - \langle \nabla f(\tilde{\bm x}^*), \bm x_t - \tilde{\bm x}^*  \rangle - \frac{\mu}{2} \| \tilde{\bm x}^* - \bm x_t \|_2^2.
\]
That is, we have the parameter estimation error
\begin{align*}
\| \tilde{\bm x}^* - \bm x_t \|_2 &\leq \sqrt{\frac{2 (\langle \nabla f(\tilde{\bm x}^*), \tilde{\bm x}^* - \bm x_t \rangle + f(\bm x_t) - f(\tilde{\bm x}^*))}{\mu}} \\
&\leq \sqrt{\frac{2 (\langle \nabla f(\tilde{\bm x}^*), \tilde{\bm x}^* - \bm x_t \rangle + f(\bm x_t) - f({\bm x}^*)}{\mu}} \\
&\leq \sqrt{\frac{2 |\langle \nabla f(\tilde{\bm x}^*), \tilde{\bm x}^* - \bm x_t \rangle|}{\mu}} + \sqrt{\frac{2 h(\bm x_t)}{\mu}}.
\end{align*} 
Notice that $|\langle \nabla f(\tilde{\bm x}^*), \tilde{\bm x}^* \rangle | \leq \| \nabla f(\tilde{\bm x}^*)\|_\infty \cdot \|\tilde{\bm x}^*\|_1 \leq \sqrt{s} C \| \nabla f(\tilde{\bm x}^*)\|_\infty$. For the term $| \langle \nabla f(\tilde{\bm x}^*), \bm x_t|$, we have
\begin{align*}
| \langle \nabla f(\tilde{\bm x}^*), \bm x_t \rangle| &= C \sum_{i=0}^{t-1} \frac{2 (i+1)}{t(t+1)} \left|\left\langle \frac{ (( - \bm x_i + \nabla f(\bm x_i)/(L\eta_i))_{S_i}}{\| (- \bm x_i + \nabla f(\bm x_i))_{S_i}\|_2}, \nabla f(\tilde{\bm x}^*) \right\rangle\right| \\ 
&\leq C \sum_{i=0}^{t-1} \frac{2 (i+1)}{t(t+1)} \left\| \frac{ (( - \bm x_i + \nabla f(\bm x_i)/(L\eta_i))_{S_i}}{\| (- \bm x_i + \nabla f(\bm x_i))_{S_i}\|_2} \right\|_2 \cdot \|\nabla f(\tilde{\bm x}^*) \|_2 \\ 
&\leq C \sqrt{s} \|\nabla f(\tilde{\bm x}^*)\|_\infty \cdot \sum_{i=0}^{t-1} \frac{2 (i+1)}{t(t+1)} 1 \\
&= C \sqrt{s} \|\nabla f(\tilde{\bm x}^*)\|_\infty,
\end{align*}
where $s$ is the maximal allowed sparsity in $\mathbb{M}$, that is, $s = \max_{S\in \mathbb{M}} |S|$. Hence, we have
\[
\| \tilde{\bm x}^* - \bm x_t \|_2 \leq \sqrt{\frac{2 \sqrt{s}C \| \nabla f(\tilde{\bm x}^*)\|_\infty}{\mu}} + \sqrt{\frac{2 h(\bm x_t)}{\mu}}.
\]

The remaining part is to show $f$ with the associated sensing matrix $\bm A$ satisfies $\mu$-quadratic growth condition with high probability. Notice that $a_{i j } \sim \mathcal{N}(0,1/\sqrt{n})$ and $e_i \sim \mathcal{N}(0,\sigma^2)$ and we know that with a large probability $\nabla f(\tilde{\bm x}^*)_i = (\bm A^\top \bm e)_i \leq 2 \sigma \sqrt{\frac{\log d}{n}}$ (See Section 4 of \citet{jain2014iterative}). More specifically, $P(\| \nabla f(\tilde{\bm x}^*)\|_\infty > 2 \sigma \sqrt{\log d / n} )  = P(\sup_j \{\bm A^\top \bm e)_j\} > 2 \sigma \sqrt{ \log d / n}) \leq  (d e^{-\sigma^2 \log d / n} \sqrt{n}) / (\sigma \sqrt{\pi \log d})$ (See Equ. 3.1 of \citet{candes2007dantzig}). 

Furthermore, $f$ defined in graph-structured linear sensing problem satisfies the $\mu$-quadratic growth condition, i.e. $\mu \geq \frac{1}{2} - \frac{c_1 s\log d}{n}$ with probability at least $1 - e^{c_0 n}$ where $c_0$ and $c_1$ are two universal constants.
Therefore, we continue to have
\begin{align*}
\| \tilde{\bm x}^* - \bm x_t \|_2 &\leq \sqrt{\frac{2 \sqrt{s}C \| \nabla f(\tilde{\bm x}^*)\|_\infty}{\mu}} + \sqrt{\frac{2 h(\bm x_t)}{\mu}} \\
&\leq \left( \sqrt{\frac{4 \sigma C \sqrt{s \log d / n}}{\frac{1}{2}} - \frac{c s \log d}{n}} + \sqrt{\frac{2 h(\bm x_t)}{\frac{1}{2} - \frac{c s \log d}{n}}} \right),
\end{align*}
where $c=c_1$ and the last inequality is valid with high probability.
\end{proof} 

\section{Adversarial examples at $\bm x_t$}
\label{appendix:adv-examples}
\begin{figure}[H]
\centering
\includegraphics[width=.4\textwidth]{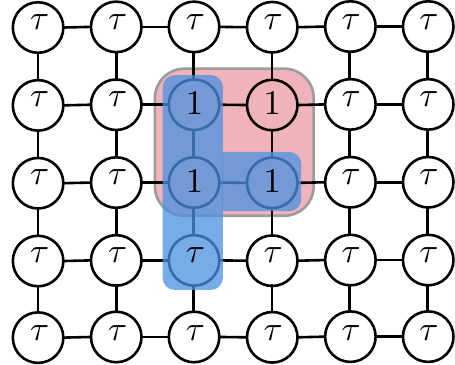}
\caption{A toy example of $S^*$ (4 nodes with 1 entry) with $s=4$ and $d=30$ where $\mathbb{G}[S^*]$ is the connected subgraph with up to 4 nodes (red region). Entries on nodes represent $\nabla f(\bm  x_t)$. This example of $\nabla f(\bm x_t)$ provides an optimal solution $ \langle \bm v_t, \nabla f(\bm x_t) \rangle = -\sqrt{4}$. The blue region is the best approximate solution with value $\langle \bar {\bm  v}_t, \nabla f(\bm x_t) \rangle = -\sqrt{3 + \tau^2}$.}
\label{fig:wrapfig}
\end{figure}

\paragraph{Adversarial example setup.} Consider the graph $\mathbb G$ illustrated in Figure \ref{fig:wrapfig}.  Suppose $f(\bm x) = \bm x^\top \bm x/2 - \bm x^\top \bm b$. Clearly, $f(\bm x)$ is $1$-strongly convex and $1$-smooth. Let $\bm b$ and $\bm x_t$ be such that $\nabla f(\bm x_t)^\top = [1, 1, 1, \ldots, 1, \tau, \tau, \ldots, \tau]^\top \in \mathbb{R}^{1\times d}$ where the first $s=4$ entries are 1 and rest $d-s= 30 - 4 = 26$ entries are $\tau$ with $0< \tau < 1$ (This is always possible if $t=0$ and one only needs to pick up such $\bm x_0$ given a predefined $\bm b$). Assume further that
$\mathbb{M} = \mathbb{M}(\mathbb{G}, s=4, g=1)$ the set of connected components of $\mathbb{G}$ each with at most $s=4$ nodes, and assume that nodes are numbered such that $S^*=\{1,2,3,4\}\in \mathbb{M}$. Then, the optimal solution is
\[
\bm v_t =\argmin{\bm v\in \mathcal D(1,\mathbb M)} \, \bm v^\top\nabla f(\bm x_t) =  \frac{1}{\sqrt{4}}(-1,-1,-1,-1,0,\ldots,0),
\]
where ${{\bm v}_t}^\top\nabla f(\bm x_t) = -\sqrt{4} = -2$.
However, the best non-optimal-support LMO will select $\bar {\bm  v}_t$ with  at least one $\tau$ entry, and thus 
$\bar {\bm  v}_t^\top\nabla f(\bm x_t) \leq \sqrt{3 + \tau^2}$.

\paragraph{Gap-additive adversarial example.} The key problem with the gap-additive assumption is the requirement for the gap error to decay. Recall $\bar{g}_t (\bm x_t) = \langle \nabla  f(\bm x_t), \bar{\bm v}_t - \bm x_t \rangle, g_t(\bm x_t) = \langle \nabla f(\bm x_t), \bm v_t - \bm x_t \rangle$ In particular, $g_t(\bm x_t) - \bar{g}_t({\bm x}_t) = \nabla f(\bm x_t)^\top(\bm v^*_t-\bar{\bm v}_t) = 2 -\sqrt{3 +\tau^2} > 0$, which is strictly positive and constant in $t$. Thus, any additive gap assumption with decaying tolerance will eventually require an exact LMO-support recovery.

\paragraph{Gap-multiplicative adversarial example.} We further consider adversarial examples for satisfying (\ref{inequ:approx-lmo-2}), continuing  the above example.
In this scenario, (\ref{inequ:approx-lmo-2}) requires 
\[
\underbrace{ (-\sqrt{3 + \tau^2} - \langle \bm x_t, \nabla f(\bm x_t)\rangle) }_{A}\leq \delta\underbrace{(-\sqrt{4} - \langle \bm x_t, \nabla f(\bm x_t)\rangle).}_{B}
\]
But for any $0< \tau < 1$, suppose that
\[
\bm x_t = \frac{1}{\sqrt{3 +\tau^2}} [-1,-1,-1, -\tau,0,\ldots,0]^\top
\Rightarrow \sqrt{3 +\tau^2} < -\langle \bm x_t, \nabla f(\bm x_t)\rangle < \sqrt{4}.
\]
To see the above, the norm of $\bm x_t$ is unit, i.e. $\|\bm x_t\|_2 = 1$ and $-\langle \bm x_t, \nabla f(\bm x_t)\rangle = \frac{3 + \tau}{\sqrt{3 + \tau^2}}$. Notice further that $\sqrt{3 +\tau^2} < -\langle \bm x_t, \nabla f(\bm x_t)\rangle = \frac{3 + \tau}{\sqrt{3 + \tau^2}} < \sqrt{4}$. Then  $A > 0$ but $B < 0$, and \emph{no positive value of $\delta$} can possibly satisfy (\ref{inequ:approx-lmo-2}); that is, the assumption is only satisfied if $\delta = 1$ and the LMO is exact, which is NP-hard.

\section{Experimental details of Fig. \ref{fig:least-square-k-support-norm}}
\label{appendix-k-support-norm}

\begin{algorithm}
\caption{Approximate ($\delta, \nabla f(\bm x_t), \mathcal{D}$)-DMO for $k$-support-norm ball.}
\begin{algorithmic}[1]
\STATE \textbf{Input}: approximation factor $\delta \in (0,1]$,  input vector $\nabla f(\bm x_t)$
\STATE $S^* = \arg\max_{S \in \mathbb{M}} \| \nabla f(\bm x_t)_{S}\|_2$
\IF{$\delta = 1$}
\STATE \textbf{return} $S^*$
\ENDIF
\STATE $\bar{S} = \arg\min_{S \in \mathbb{M}} \| \nabla f(\bm x_t)_{S}\|_2$
\FOR{$j \in S^*$}
\STATE randomly remove an element from $\bar{S}$
\STATE $\bar{S} = \bar{S} \cup {j}$
\IF{$\|\nabla f(\bm x_t)_{\bar{S}}\|_2 \geq \delta \|\nabla f(\bm x_t)_{S^*}\|_2$}
\STATE \textbf{return} $\bar{S}$
\ENDIF
\ENDFOR
\end{algorithmic}
\label{alg:k-support-norm-lmo-approx}
\end{algorithm}

In Fig. \ref{fig:least-square-k-support-norm}, we consider the following optimization problem
\[
\min_{\bm x \in \mathcal{D}} f(\bm x) := \| \bm A \bm x - \bm b \|_2^2,
\]
where $\mathcal{D}:= \{ \bm x: \operatorname{supp}(\bm x) \in \mathbb{M}, \|\bm x\|_2 \leq C\}$ with $C=1$ and $\mathbb{M} = \{ S\subseteq [d]: |S|\leq s=5\}$. The LMO operator for this norm ball can be calculated as the following
\begin{equation}
\bm v_t = \arg\min_{\bm x \in \|\bm x\|_2 \leq 1} \left\langle \nabla f(\bm x_t), \bm x \right\rangle = \frac{- \nabla f(\bm x_t)_{S^*}}{\|\nabla f(\bm x_t)_{S^*}\|_2}, \quad S^* \in \arg\max_{S \in \mathbb{M}} \| \nabla f(\bm x_t)_S\|_2. \label{equ:exact-operator}
\end{equation}

\begin{figure}[H]
\centering
\includegraphics[width=.95\textwidth]{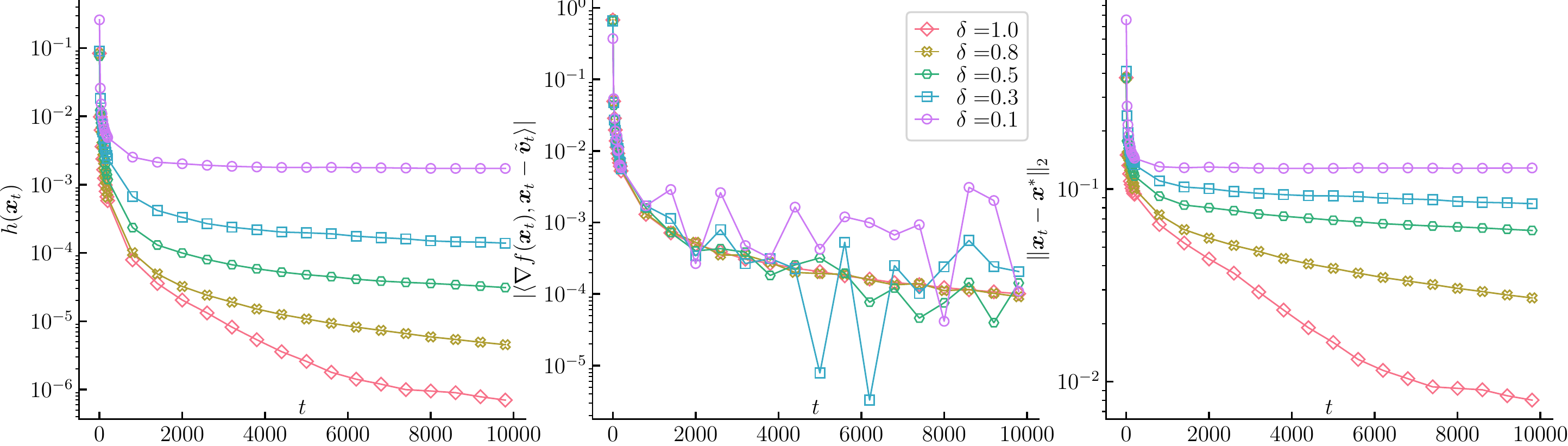}
\caption{\textit{Left}: The primal error $h(\bm x_t)$ as a function of $t$ for \textsc{DMO-FW-I} with different $\delta$. \textit{Middle}: The absolute value of $|\langle \nabla f(\bm x_t), \bm x_t - \tilde{\bm v}_t \rangle|$. \textit{Right}: The estimation error $\| \bm x_t - \bm x^*\|_2$ as a function of $t$ for \textsc{DMO-FW-I} with different $\delta$. The number of measurements $n=200$, i.e. $\bm A \in \mathbb{R}^{200\times d}$.}
\label{fig:dmo-fw-II-lasso}
\end{figure}

To illustrate our approximation bound, instead of using the above exact operator (\ref{equ:exact-operator}), we obtain an approximate DMO operator for the above optimization problem and present the $\delta$-approximate DMO in Alg. \ref{alg:k-support-norm-lmo-approx}. It returns an $S$ such that is at least $\delta$-approximation DMO operator oracle. The key step to control the quality of $S$ is Line 6 where $S$ is returned whenever $\|\nabla f(\bm x_t)_S\|_2 \geq \delta \|\nabla f(\bm x_t)_{S^*}\|_2$ and $S^*$ is the subset of maximal magnitudes of $\nabla f(\bm x_t)$. 

Our experimental setting is as follows: We use a normalized Gaussian sensing matrix where each entry $A_{i j}\sim \mathcal{N}(0,1/\sqrt{n})$. The number of samples $n=200$ and the dimensionality $d=500$. We fix sparsity $s=50$ for $\bm x^*$, i.e. $\| \bm x^*\|_0 = 50$ where each nonzero entry is either 1 or -1 with same probability. We obtain $\bm b =\bm A \bm x^*$ with $\bm x^* = \bm x^* / \|\bm x^*\|_2$. We also plot the duality gap $|\langle \nabla f(\bm x_t), \bm x_t - \tilde{\bm v}_t \rangle|$ and estimation error of $\bm x^*$, i.e., $\| \bm x^t - \bm x^*\|_2$. 

\section{More experimental details}
\label{appendix:section:experiments}

\subsection{Experimental setup}
\label{experimental-setup}

\paragraph{Parameters of all methods.}
All methods except for CoSAMP use the same approximation operator, i.e., head projection proposed in \cite{hegde2015nearly} (See details in Sec. \ref{appendix:section:dmo}). \textsc{CoSAMP} uses the $s$-sparse thresholding operator. \textssc{GraphCoSAMP} share the same parameter setting as \textsc{CoSAMP} but uses graph projection operator. \textssc{Gen-MP} has $L$-smooth parameter where we estimate $L$ by finding the largest eigenvalue of $\bm A^\top \bm A$. The step size of \textssc{Graph-IHT} is then set to $\eta_t = 1/L$. The step size of both \textssc{DMO-FW} and \textssc{DMO-AccFW} are set to $\eta_t = 2/(t+2)$ for all $t \geq 0$.

\paragraph{Datasets.} In our experiments, we use two datasets: 1) 10 MNIST images. We randomly select 10 MNIST images as our graph-structured signals $\bm x^* \in \mathbb{R}^{28\times 28}$ and normalized them into a unit vector; 2) Angio image. We also choose a sparse angio image from \cite{hegde2015nearly} where $\bm x^* \in \mathbb{R}^{100\times 100}$. These sparse images have 1 connected component.  We run all methods on a sever with 246GB memory and 80 cores. All methods are implemented in Python-3.8. The graph projection operator is implemented in C++11.

\paragraph{Graph-structured sparse recovery.} The goal of GS sparse recovery is to recovery a sparse image $\bm x^*$ with several small connected components as a prior. For example, in Angio image, we consider this $d=100\times 100$ sparse image where the true image $\bm x^*$ is shown in \ref{fig:sparse-image-recovery} (bottom left). The underlying sparsity pattern has $g=1$ connected components. We then normalize $\bm x^*$ such that $\|\bm x^*\|=1$. Measurements $\bm y$ are generated by $\bm y = \langle \bm A, \bm x^* \rangle + \bm e$ where $\bm e \sim \epsilon \cdot \mathcal{N}(\bm 0, \bm I_d)$ and $\epsilon$ controls the magnitude of noise $\bm e$. To summarize, the objective in our experiment is $\min_{\bm x \in \mathcal{D}} f(\bm x) = \frac{1}{2}\left\| \bm A \bm x - \bm y \right\|_2^2$, where $\mathcal{D} = \operatorname{conv}\left\{\bm x \in \mathbb{R}^d : \|\bm x \|_2 \leq 1, {\rm supp}({\bm x}) \in \mathbb{M} \right\}, \text{ and }$, $\mathbb{M} = \{F = S_1 \cup S_2 \cup \cdots \cup S_{11}: \text{$S_i$ are CCs of $\mathbb G$},\;  |F|\leq s \}$. $\bm A \in \mathbb{R}^{n\times d}$ is a Gaussian sensing matrix where each entry $a_{i j } \sim \mathcal{N}(0, 1/\sqrt{n})$ independently. We run each experiment for 20 trials.

\subsection{More results}

In the MNIST image recovery task, we set $n = 5 \cdot|\operatorname{supp}(\bm x^*)|$ where $\bm x^*$ is a specific normalized MNIST image. To compare \textsc{DMO-FW} with \textsc{DMO-AccFW} on all ten sparse MNIST images. The prime error $h(\bm x_t)$ as a function of time $t$ is illustrated in Fig. \ref{fig:sparse-image-loss}. These results indicate that \textsc{DMO-AccFW} is \textsc{DMO-FW} on all of these sparse images. Similarly, Fig. \ref{fig:sparse-image-estimation} presents the estimation errors $\|\bm x_t - \bm x^*\|_2$ over time $t$. 

\begin{figure}[H]
\centering
\includegraphics[width=.95\textwidth]{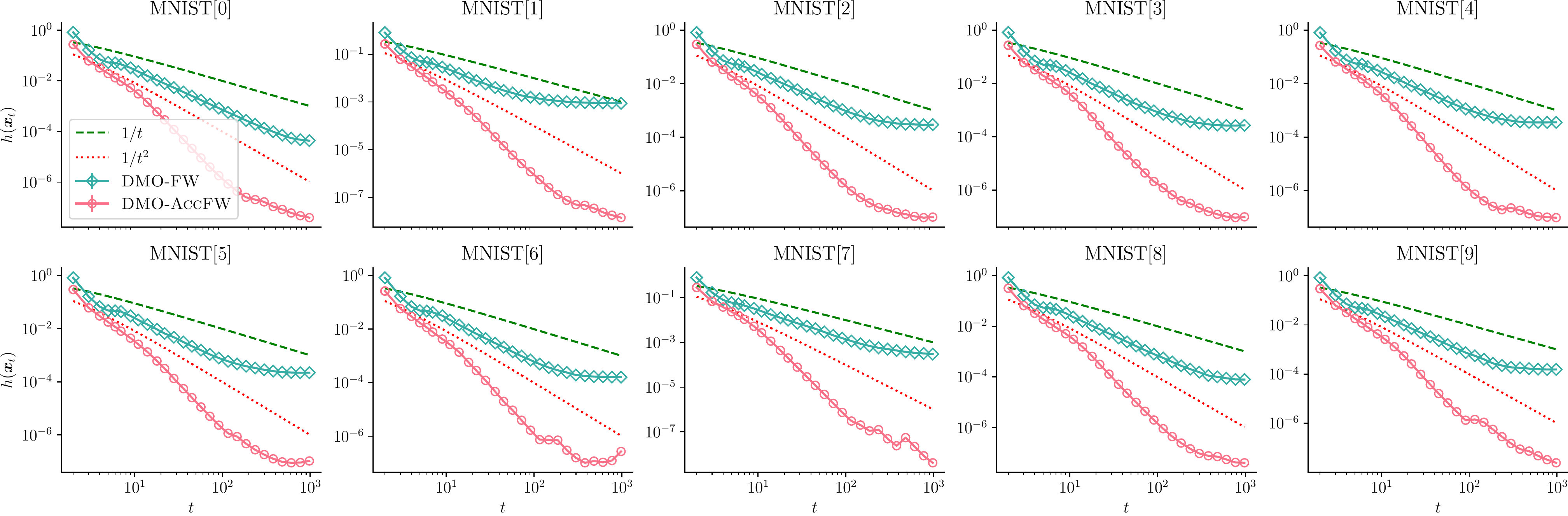}
\caption{Primal error $h({\bm x}_t)$ as a function of $t$ on task of the graph-structured sparse recovery of ten MNIST images.}
\label{fig:sparse-image-loss}
\end{figure}

\begin{figure}[H]
\centering
\includegraphics[width=1.\textwidth]{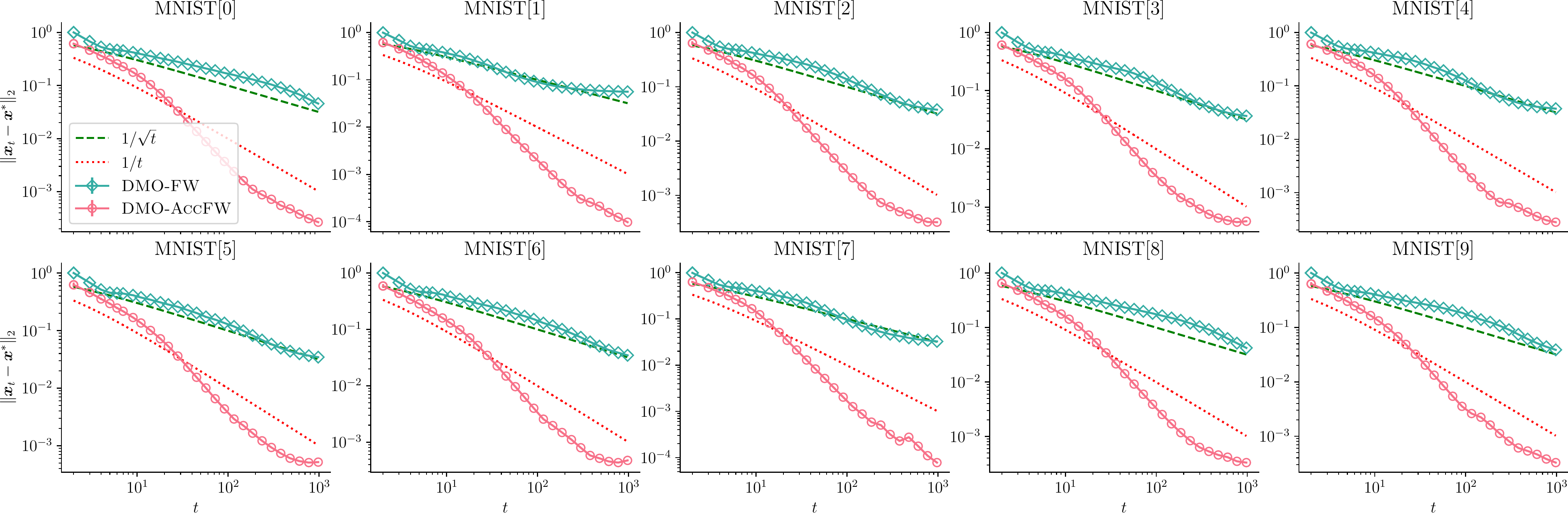}
\caption{Estimation error $\| \bm x_t - \bm x^*\|_2$ as a function of $t$ on task of the graph-structured sparse recovery of ten MNIST images.}
\label{fig:sparse-image-estimation}
\end{figure}

\begin{figure}[H]
\centering
\includegraphics[width=.8\textwidth]{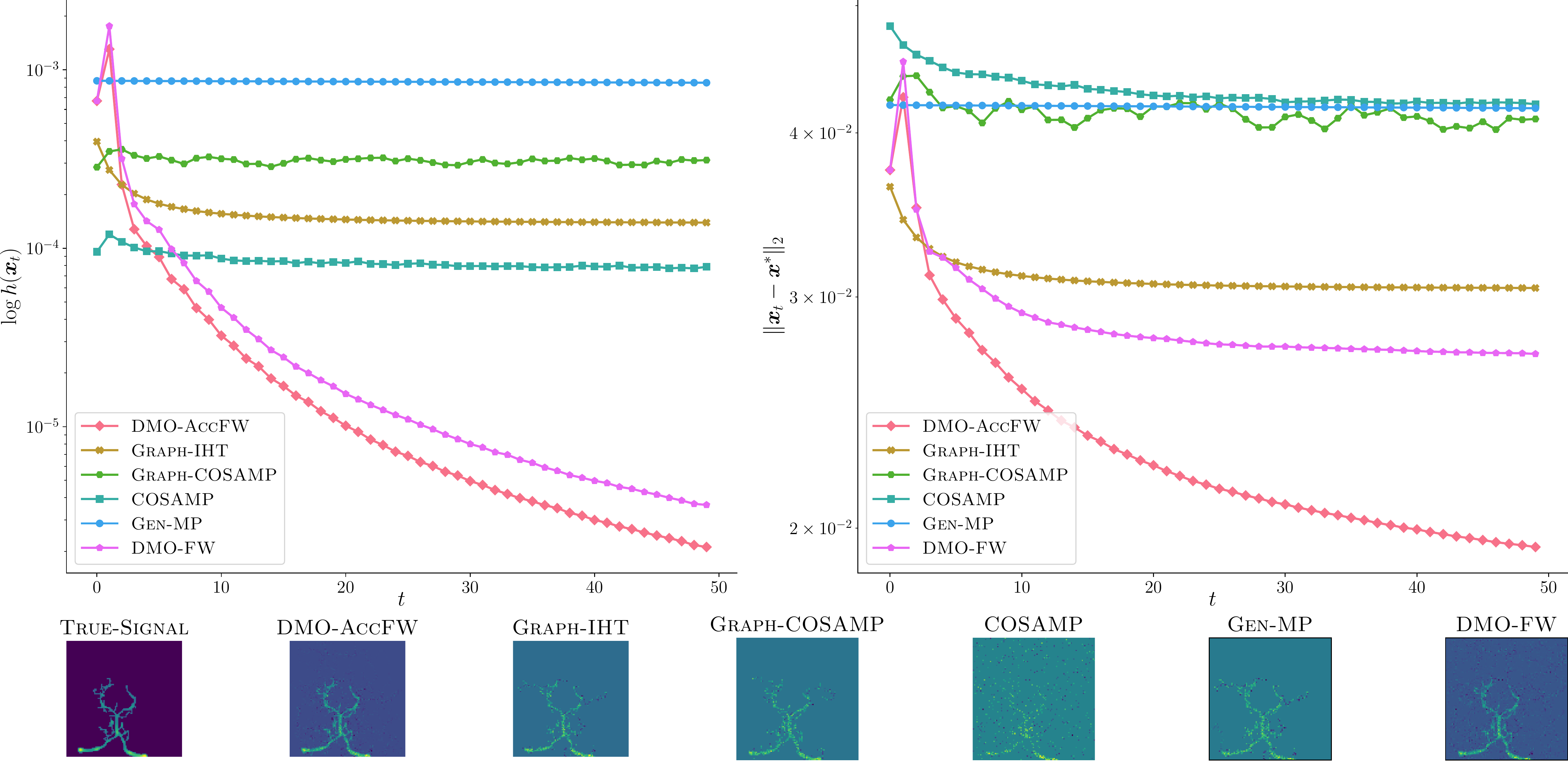}
\caption{The performance of methods on the task of graph-structured linear regression task. \textit{Top}: 
\textsc{DMO-AccFW} vs other baseline methods on primal error $ h(\bm x_t)$ (\textit{left}) and variable suboptimality $\|\bm x_t-\bm x^*\|_2$ (\textit{right}) as a function of $t$. 
\textit{Bottom}:  
Recovered sparse images $\bm x_t$ vs truth image (bottom right) after $50$ iterations. \vspace{-4mm}}
\label{fig:more-results-2}
\end{figure}

\paragraph{Performance of \textsc{DMO-FW}.} We include the results of \textsc{DMO-FW} where its performance is between \textsc{GraphIHT} and \textsc{DMO-AccFW}. One may notice that \textsc{GraphIHT} is stuck as a local minimum in Fig. \ref{fig:sparse-image-recovery} since it solves a nonconvex problem. It is a known issue with IHT methods, regardless of step size. The Fig. \ref{fig:more-results-2} on the right (gold line) demonstrates the same stalling behavior when $\eta_t = 1/(t + 1)$ is used. 

\begin{figure}[H]
\centering
\includegraphics[width=.95\textwidth]{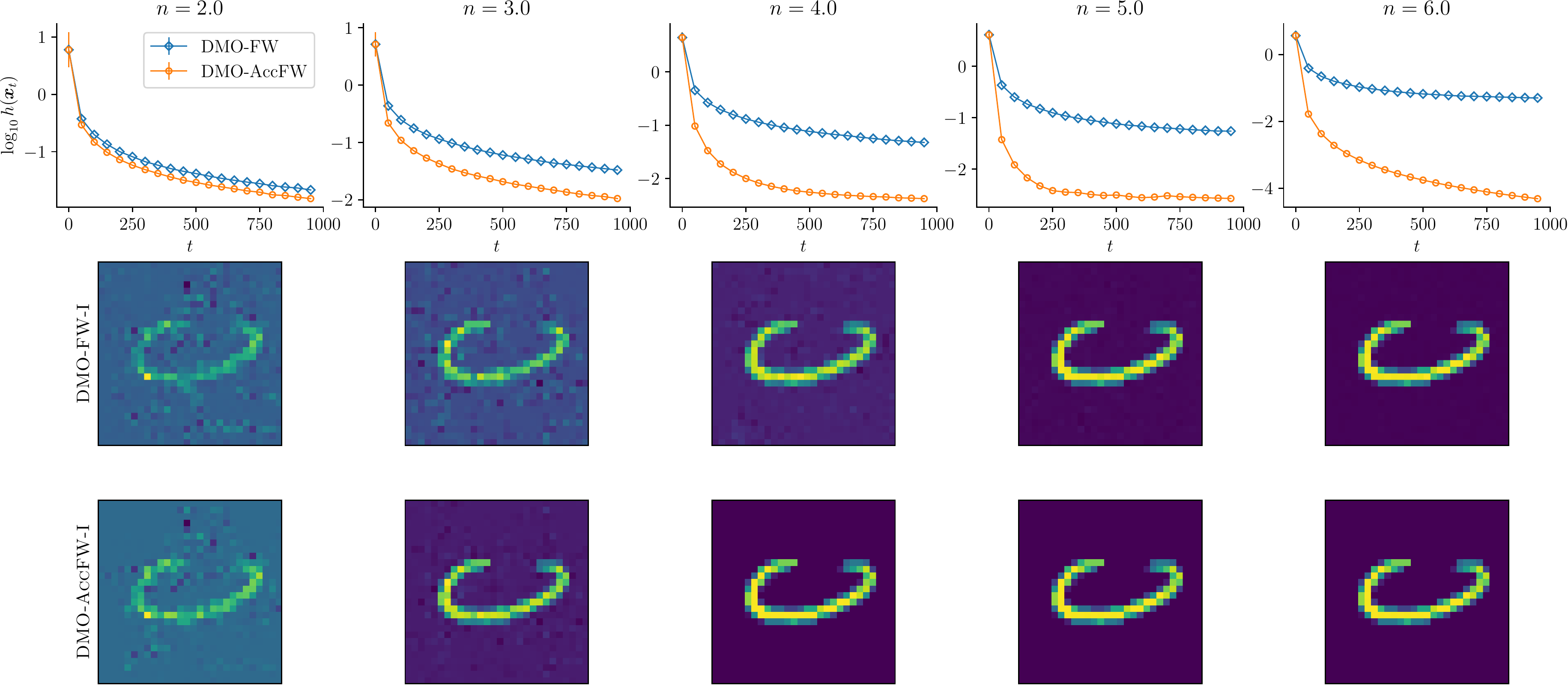}
\caption{The comparison of \textsc{DMO-FW-I} with \textsc{DMO-AccFW-I} for different sampling ratio.}
\label{fig:fw-dmo-vs-accfw-dmo-II-1000}
\end{figure}

Fig. \ref{fig:fw-dmo-vs-accfw-dmo-II-1000} showcases the learning process as a function of training ratio from $2.0$ to $6.0$.

\section{Dual maximization oracles}
\label{appendix:section:dmo}
\paragraph{DMO via a heuristic method.} A heuristic method with $\delta = \sqrt{1/\lceil s/g\rceil}$ for $\mathbb{M}(\mathbb{G},s,g)$. In Section \ref{section:dmo:inner-product}, we present a heuristic procedure that is for $\mathbb{M}(\mathbb{G},s,g)$. Algorithm \ref{algo:heuristic-dmo} presents this heuristic method, which has three main steps: \textbf{Step 1)} Let $I_g$ be the indices of $g$ largest magnitude $|z_i|$. Initialize a node set as $S = I_g$ (Line 2 and Line 3); \textbf{Step 2)} Next, iterate through the edges $(u,v) \in \mathbb{E}$, in any order. If $u\in S$, merge $v$ into $S$; similarly, if $v\in S$, merge $u\in S$. If at any point $|S| = s$, terminate;
\textbf{Step 3)} Repeat Step 2 until either no new edges are added, or $|S| = s$ (Line 10 to Line 24). This procedure finds a $\delta$-DMO for $\mathbb{M}(\mathbb{G},s,g)$ with $\delta = 1/\lceil s/g\rceil$, with runtime linear to the number of edges $\mathcal{O}(m)$. 

We prove that $S$ returned by Algorithm \ref{algo:heuristic-dmo} satisfies $\| \bm z_S\|_* \geq 1/\lceil s /g \rceil \max_{S^\prime\in \mathbb{M}}\|\bm z_{S^\prime}\|_*$:  First of all, $S$ is in $\mathbb{M}$ and notice that $\| \bm z_{S^\prime}\|_*^2 = \sum_{i \in I_g} |z_i|^2 + \sum_{j\in S^\prime\backslash I_g} |z_j|^2$, $\forall S^\prime \in \mathbb{M}$. As $S$ contains $g$ largest magnitudes of $\bm z$, we have $(\lceil s/g \rceil - 1) \sum_{i \in S} |z_i|^2 \geq \sum_{j\in S^\prime\backslash I_g} |z_j|^2$. This inequality provides $\lceil s/g \rceil \| \bm z_S\|_*^2 \geq \sum_{i \in I_g} |z_i|^2 + \sum_{j\in S^\prime\backslash I_g} |z_j|^2 = \|\bm z_{S^\prime}\|_2^2$. Hence, we have $\lceil s/g \rceil \| \bm z_S\|_*^2 \geq\max_{S^\prime \in \mathbb{M}} \|\bm z_{S^\prime} \|_*^2$. Taking square root of both sides will provide a better approximation guarantee, i.e. $\delta = 1/\lceil s/g \rceil \geq \delta^\prime = \sqrt{1/\lceil s/g \rceil}$. Clearly, the total run time is $\mathcal{O}(m)$ dominated by the for loop of Line 12.

\begin{algorithm}[H]
\caption{\textsc{A heuristic DMO with $\delta=\sqrt{1/\lceil s/g\rceil}$} approximation guarantee}
\begin{algorithmic}[1]
 \STATE \textbf{Input}: underlying graph $\mathbb{G}$, sparsity $k$, number of CCs $g$, input vector $\bm z$
 \STATE Sort entries of $\bm z$ by magnitudes such that $|z_{\tau_1}| \geq |z_{\tau_2}| \geq \ldots \geq |z_{\tau_g}|\geq |z_{\tau_{g+1}}|$ \algorithmiccomment{Notice that this step can be done in $\mathcal{O}(d)$ time by using Floyd-Rivest selection algorithm \citep{floyd1975algorithm}.}
\STATE $I_g = [\tau_1,\tau_2,\ldots,\tau_g], S = I_g$
\STATE $\bm c = \bm 0$ \algorithmiccomment{Initially, all nodes have same connected component ID}
\STATE $i=1$ // Tracking the ID of connected component
\FOR{$v \in S$}
\STATE $c_{v} = i$ // Node $v$ has a component ID $i$
\STATE $i = i +1$ 
\ENDFOR
\STATE $\mathbb{F}=\emptyset$ // Keep edges that are in $g$ components
\IF{$|S| = s$} 
\STATE \textbf{Return} $S$ // We assume $g \leq s$
\ENDIF
\FOR{$(u,v) \in \mathbb{E}$}
\IF{$c_{u} == 0$ and $c_{v} \ne 0$}
\STATE $S = S \cup \{u\}$
\STATE $\mathbb{F} = \mathbb{F} \cup (u,v)$
\STATE $c_u = c_v$ // $u$ is added to $c_v$-th component
\ENDIF
\IF{$|S| = s$}
\STATE \textbf{Return} $S$
\ENDIF
\IF{$c_{u} \ne 0$ and $c_{v} == 0$}
\STATE $S = S \cup \{v\}$
\STATE $\mathbb{F} = \mathbb{F} \cup (u,v)$
\STATE $c_v = c_u$ // $v$ is added to $c_u$-th component
\ENDIF
\IF{$|S| = s$}
\STATE \textbf{Return} $S$
\ENDIF
\ENDFOR
\end{algorithmic}
\label{algo:heuristic-dmo}
\end{algorithm}

\paragraph{DMO via the head projection operator.} \citet{hegde2015nearly} presents an algorithm for $\mathbb{M}(\mathbb{G},s,g)$ that has $\delta=\sqrt{1/14}$. We state a simplified version of it as the following: Consider $\mathbb{M}(\mathbb{G}, s, g)$-WGM and let ${\bf z}\in \mathbb{R}^d$. Then there is an algorithm that returns a support $S \subseteq [d]$ in $\mathbb{M}(\mathbb{G}, 2 s + g, g)$-WGM satisfying that $\| {\bf z}_S \|_2 \geq \delta \cdot \max_{S'\in \mathbb{M}} \| {\bf b}_{S'} \|_2$, where $\delta = \sqrt{1/14}$ and it runs in $\mathcal{O}(m \log^3 (d))$ where $m$ is the number of edges in $\mathbb{G}$.

The operators have a budget $B = s-g$. The budget value is 1 for edge cost in our experiments. In this case, the cost budget will never be violated since total costs in a $g$ forest are always not greater than $s-g$. The essential idea of this operator is a binary search over the Price-Collecting Steiner Forest problem \citep{hegde2014fast-pcst}. It then prunes over the final forest so that the returning $\mathbb{G}[S]$ is ``dense''. A C++ implementation of PCSF-GW is publicly available at \url{https://github.com/ludwigschmidt/cluster_approx}. In our experiments, we implement a C-version, which is marginally faster.
\subsection{Other graph-structured models}
{\renewcommand{\arraystretch}{1.5}

\begin{table}[H]
\centering
\caption{DMOs of different $\mathbb{M}$. DP is for Dynamic Programming.}
\begin{tabular}{p{0.28\textwidth}|p{0.3\textwidth}|p{0.12\textwidth}|p{0.15\textwidth}}
\toprule
$\mathbb{M}$ & DMO & Complexity & $\delta$-approx.\\\hline 
$\left\{H:H := \cup_{i=1}^s S_s\right\}$ where graph $\mathbb{G}$ is a tree and $S_k$ is a subtree  & Tree decomp. \citep{lim2017k} &  $\mathcal{O}(m s + d)$ & $\delta=1$ \\\hline
$\left\{ S: \mathbb{T}_S \text{ is a subtree. } |S|\leq s\right\}$ where $\mathbb{T}$ is a tree and $\mathbb{G}= \mathbb{T}$  & DP \cite{hochbaum1994node} & $\mathcal{O}(s^2 d)$ & $\delta=1$ \\\hline
$\mathbb{M}(\mathbb{G},s,g)$  & Algorithm \ref{algo:heuristic-dmo} & $\mathcal{O}(m)$ & $\delta = \sqrt{1/\lceil s/g\rceil}$ \\\hline
$\mathbb{M}(\mathbb{G},s,g)$  & Head Proj. \citep{hegde2015nearly} & $\mathcal{O}(m \log^3 d)$ &  $\delta=\sqrt{1/14}$ \\\bottomrule
\end{tabular}
\label{tab:popoular-model-m}
\end{table}}

\paragraph{Other operators and applications.} We list GS models in Table \ref{tab:popoular-model-m} with time complexities and approximation guarantees. These operators consider connectivity constraints, a key property or requirement of subgraph detection. Connectivity and subgraph detection have been explored recently \citep{arias2011detection,qian2014connected,hegde2015approximation,aksoylar2017connected}. For example, if we assume $\mathbb{M} = \{S: |S|\leq s, \mathbb{G}[S] \text{ is connected.}\}$, DMO operator can be reformulated as $s$-maximum-weight subgraph problem, which has been considered in \cite{hochbaum1994node}. This algorithm has been applied to identify subnetwork markers in protein-protein interaction (PPI) network \citep{dao2011optimally} and automatic planning \citep{riabov2006scalable}.

\subsection{Comparison of Convergence rate}
\label{appendix:tab:convergence-rate}
This subsection summarizes and compares the convergences rate of different method as presented in Table \ref{tab:convergence-rate}.

\renewcommand{\arraystretch}{1.4}
\begin{table*}
\caption{Comparison of convergence rates of  FW-type methods with ours when $\mathcal D$ is a GS support set. In all cases, we assume the diameter of $\mathcal D$ is $D := \max_{\bm x, \bm y\in \mathcal{D}} \| \bm x - \bm y\|$ and $f$ is convex differentiable. The column operator availability checks whether an approximate LMO/DMO-operator is efficiently obtainable for GSCOs. In our problem $D = 4 C^2$.}
\centering
\begin{tabular}{p{0.18\textwidth}|p{0.1\textwidth}|p{0.08\textwidth}|p{0.33\textwidth}|p{0.15\textwidth}}
\toprule
\centering
 Algorithm & Operator Availability & Solution & Condition & Convergence rate \\\hline 
Inexact gap-additive \citep{jaggi2013revisiting} & \textcolor{red}{\centering \xmark} & $\bm x_t \in \mathcal{D}$ & $L$-smooth & $\frac{2 L D^2 (1+\delta)}{ (t+2)}$  \\\hline 
 Inexact gap-mult. \citep{pedregosa2020linearly} & \textcolor{red}{\centering \xmark} & $\bm x_t \in \mathcal{D}$ & $L$-smooth &  $\frac{2(L D^2 + B D \delta)}{(\delta^2 t + 2 \delta)}$ \\\hline 
\textsc{DMO-FW-I} & \textcolor{ForestGreen}{\centering \cmark} & $\bm x_t \in \mathcal{D}$ & $L$-smooth, $\|\nabla f(\bm x)\|_\infty\leq B/t^\nu$ & $\mathcal{O}\left( \frac{B C \sqrt{s}}{t^\nu} \right)$ \\\hline 
\textsc{DMO-FW-II} & \textcolor{ForestGreen}{\centering \cmark} & $\bm x_t \in \frac{\mathcal{D}}{\delta}$ & $L$-smooth & $\frac{2 L D^2}{\delta^2 (t+2)}$ \\\hline
\textsc{DMO-AccFW-I} & \textcolor{red}{\centering \xmark} & $\bm x_t \in \mathcal{D}$ & $L$-smooth, $\mu$-quadratic, $\|\bm x^*\|_2 = C$ & $\frac{4 e^{4L /\mu} h(\bm x_0)}{(t+2)^2}$ \\\hline
\textsc{DMO-AccFW-I} & \textcolor{ForestGreen}{\centering \cmark} & $\bm x_t \in \mathcal{D}$ & $L$-smooth, $\|\nabla f(\bm x)\|_\infty \leq B/t^\nu$ & $\mathcal{O}\left( \frac{B C \sqrt{s}}{t^\nu} \right)$ \\\hline
\textsc{DMO-AccFW-II} & \textcolor{ForestGreen}{\centering \cmark} & $\bm x_t \in \frac{\mathcal{D}}{\delta}$ & $L$-smooth, $\mu$-quadratic growth & $\mathcal{O}\left( \frac{L^2 (C^2/\delta^2 - D_*^2) }{\mu(t+2)} \right)$ \\\bottomrule
\end{tabular}
\label{tab:convergence-rate}
\end{table*}

\end{document}